\documentclass{article}
\usepackage[utf8]{inputenc}
\usepackage[bottom]{footmisc}
\usepackage{authblk}

\usepackage[margin=1in]{geometry} % uncomment for defaut margins

\usepackage{amsmath,amssymb,amsthm}
\usepackage{stmaryrd}
\usepackage{graphicx}
\usepackage{subcaption}
\usepackage[]{algorithm2e}
\usepackage{caption}

\usepackage{tikz}
\usepackage{bbm}
\usepackage{tikz-cd}
\usetikzlibrary{decorations.markings}
\usepackage{hyperref}
\usepackage{url}
\usepackage{accents}
\usepackage{mathrsfs}
\usepackage[]{tikz-cd}
\usepackage[all]{xy}
\usepackage{multicol}
\usepackage{mathtools}
\usepackage{color}
\usepackage[normalem]{ulem}

\newtheorem{Lemma}{Lemma}
\newtheorem{Corollary}[]{Corollary}

\theoremstyle{definition}
\newtheorem*{Claim*}{Claim}
\newtheorem*{Lemma*}{Lemma}
\newtheorem{Example}{Example}

\newtheorem{Remark}{Remark}

\newtheorem{Definition}{Definition}
\newtheorem{Proposition}{Proposition}
\newtheorem{Theorem}{Theorem}

\newcommand{\U}{\mathcal{U}}

\newcommand{\W}{\mathcal{W}}
\renewcommand{\P}{\mathcal{P}}

\renewcommand{\r}{\mathbb{R}}

\newcommand{\ip}[2]{\langle {#1}, {#2} \rangle}
\renewcommand{\o}{\mathcal{O}}
\renewcommand{\|}{|}

\renewcommand{\i}{\mathscr{I}}

\newcommand{\im}{\mathrm{im}}

\renewcommand{\c}{\mathbb{C}} 

\renewcommand\xleftrightarrow[2][]{\ext@arrow 0099{\longleftrightarrowfill@}{#1}{#2}}
\def\longleftrightarrowfill@{\arrowfill@\rightarrow}
\def\longrightarrowfill@{\arrowfill@\rightarrow}
\newcommand{\rank}[1]{\mathrm{rank}\left(#1\right)}

\renewcommand{\t}{\mathcal{T}}
\newcommand{\V}{\mathcal{V}}

\renewcommand{\o}[1]{\overline{#1}}

\renewcommand{\O}{\mathrm{O}}

\renewcommand{\i}{\mathcal{I}}

\newcommand{\valid}{\texttt{valid}}
\newcommand{\counts}{\texttt{counts}}
\newcommand{\used}{\texttt{used}}
\newcommand{\usedind}{\texttt{used\_indices}}

\usepackage{algpseudocode}

\title{Euclidean distance geometry and the orthgonal beltway problem}
\author{Dan Edidin and Arun Suresh\thanks{edidind@missouri.edu, aszxy@missouri.edu}}
\affil{Department of Mathematics, University of Missouri-Columbia, MO 65211}

\begin{document}
\maketitle
\begin{abstract}
The orthogonal beltway problem is the problem of recovering the $\O(n)$-orbit of a $\delta$-function supported
at a finite number of points in $\r^n$  from its auto-correlation or, equivalently, second moment. It 
was introduced in~\cite{bendory2024beltway}
as a generalization of the classical beltway problem in X-ray crystallography. The higher dimensional version of the beltway problem was motivated by cryo-electron microscopy~\cite{bendory2023autocorrelation}.
In this paper we prove that if $m > n$, then the $\O(n)$-orbit of generic binary signal supported at $m$ points where at least $\ell$ of them 
have equal magnitude can be recovered from its auto-correlation. We also provide a connection to Euclidean distance
geometry and prove, as a corollary of our main theorem, that if $m > n$, then the $\O(n)$-orbit of a generic collection of $m$ points
on the sphere $S^{n-1}$ can be recovered from their unlabeled interpoint distances.

We take advantage of the parallels to Euclidean distance geometry and develop a polynomial-time reconstruction algorithm for recovering the $\O(n)$-orbits of binary $\delta$-functions
from their second-moment data when at least one of the points has distinct magnitude. In $\mathbb{R}^3$, the complexity of our algorithm is bounded from above by $O(m^8)$ but we show that in practice the complexity is much lower. We also demonstrate that the algorithm is robust to low levels of noise.
Finally, we extend our algorithm to successfully perform recovery when all
the support vectors lie on a common sphere, and in this case we match the time complexity of $O(m^8)$.
\end{abstract}

\section{Introduction}
\subsection{Main contributions}
The beltway problem seeks to recover $m$ points $v_1,\dots, v_m$ on the unit
circle, up to rotation and reflection, from their unlabeled interpoint distances. 
This problem first arose in X-ray crystallography where it was observed by Patterson~\cite{patterson1944ambiguities} that the auto-correlation of a binary signal in $\r^n$ is equivalent to the set of interpoint distances of the support of the signal in the group of $n$-th roots of unity.
Motivated by cryo-electron microscopy, the paper~\cite{bendory2024beltway} introduced a higher-dimensional
version of the beltway problem which they called the {\em beltway problem over orthogonal groups}. This is the problem
of recovering the $\O(n)$-orbit of a $\delta$-function supported at $m$ points in $\r^n$ from its 
second moment under the action
of the orthogonal group $\O(n)$. When the points lie on a sphere, the information determined by the second moment is exactly the set of (unlabeled) interpoint distances, so the classical beltway problem is the $\O(2)$-beltway problem on $S^1$.

In~\cite{bendory2024beltway} the authors
showed that the second moment determines the $\O(n)$-orbit of a $\delta$-function supported at a collection of $m$ points in $\r^n$ which
have distinct magnitudes. On the other hand, they also show~\cite[Proposition 5.1]{bendory2024beltway} that if the $m$ support vectors are linearly independent (so $m  \leq n$) and at least two of points have the same magnitude then the second moment does not in general determine the $\O(n)$-orbit
of a binary $\delta$-function.
The purpose of this paper is to prove that, when $m > n$, the $\O(n)$-orbit of 
a binary $\delta$-function supported at $m$-points in $\r^n$ is determined, almost surely, by its second moment, no matter how many of the points have equal magnitude.

Let $\V_{m,n}$ denote the set of binary
$\delta$-functions $\sum_{i=1}^m \delta_{v_i}$ supported at $m$-distinct points in $\r^n$. The set $\V_{m,n}$
can be parametrized as an algebraic variety which is the quotient of the open set $\U_{m,n} =(\r^n)^m \setminus \Delta \subset (\r^n)^m$
parametrizing $m$-tuples of distinct points by the natural action of $S_m$ that permutes the points.  
For $0 \leq \ell \leq m$ define $\V^{\ell}_{m,n}$ to be the closed subset 
of $\V_{m,n}$ where a subset of at least $\ell$ of the points have the same magnitude. Note that if $\ell < \ell'$ then $\V^{\ell'}_{m,n}$ is a proper closed subset of $\V^{\ell}_{m,n}$.

% Statement of the main theorem 
\begin{Theorem}\label{thm:mainthm}
If $m > n \geq 2$ then there exists a non-empty Zariski open subset $U^{\ell}$ of $\V^{\ell}_{m,n}$ such that for any $
x = \sum_{i=1}^m\delta_{v_i} \in U^\ell$
the second moment $m_2(x)$ uniquely determines the $\O(n)$-orbit of $x$.
\end{Theorem}
Viewed in tandem with Proposition 5.1 of~\cite{bendory2024beltway}, we observe that the bound of $n+1$ is sharp.  If we take $\ell =m$
(i.e., all points lie on a sphere) then 
Theorem~\ref{thm:mainthm} verifies Conjecture 5.4 of~\cite{bendory2024beltway}. 
The case $\ell = m$ of our main result also has an important application
to Euclidean distance geometry. 
\begin{Corollary} \label{cor.interpoint}
If $S = \{v_1, \ldots , v_m\}$ is a generic configuration of $m > n$ points on the sphere $S^{n-1}$
 then $S$ is determined up to orthogonal transformation from the set of unlabeled distances
 $\{|v_i  - v_j|\}_{i < j}$.
\end{Corollary}
%Specifically Corollary~\ref{cor.interpoint} states that if $m > n$ then a generic collection
%of $m$ points on the sphere $S^{n-1}$ can be recovered from their unlabeled interpoint distances.

In Section~\ref{sec:recoveringx}, we take advantage of the parallels to Euclidean distance geometry and develop a polynomial-time reconstruction algorithm for recovering the $\O(n)$-orbits of such $\delta$-functions
from their second-moment data when at least one of the points has magnitude distinct from those of the other support vectors. In $\mathbb{R}^3$, the complexity of our algorithm is bounded from above by $O(m^8)$ but we show that in practice the complexity is much lower.
%representing a significant improvement over the $O(m^{13})$ bound in~\cite{duxbury2016unassigned} for the unlabled distance problem. 
In Section~\ref{sec:unitsphere} we extend our algorithm to successfully perform recovery when all
the support vectors lie on a common sphere, 
and in this case we match the time complexity of $O(m^8)$.
Finally, in Section~\ref{sec:noisyalg}, we extend our algorithm to the setting where the second moment measurements are noisy, and we demonstrate robust performance and near perfect reconstruction accuracy in low-noise regimes.

\subsection{Connection to Cryo-EM} 
%The motivation for considering the orthogonal beltway problem comes from cryo-electron microscopy. 
Single particle cryo-electron microscopy is a leading technology used
to elucidate the structure of biological molecules~\cite{doerr2022emerging, toader2023frontier}. The goal of single particle cryo-electron microscopy is to recover the electrostatic potential function of a cryogenically frozen molecule from 
a collection of noisy measurements that are tomographic projections of the molecule rotated by unknown random elements
$SO(3)$. Because the signal-to-noise ratio in cryo-EM experiments is very low, it is not feasible to directly estimate the unknown rotations. One approach to signal recovery from cryo-EM measurements is
to use the \textit{method of moments} which takes advantage of the fact
that for a sufficiently large sample size the empirical moments of the measurements well approximate the true moments of the underlying electrostatic potential~\cite{bendory2020single, kam1980reconstruction, abas2022generalized}. 
However, the minimum number of samples required for successful estimation grows as $N\geq \sigma^{2d}$ where $d$ is the lowest order moment measurement that determines protein structure~\cite{bandeira2017optimal, perry2019sample}. 
While it known that generic molecular structures can be recovered from their (unprojected) third moments~\cite{bandeira2023estimation, bendory2025orbit}, such recovery demands a sample complexity of at least $\omega(\sigma^6)$. Thus it is beneficial to identify families of molecular structures that can be recovered from moments of order two, thereby reducing the sample complexity from $\omega(\sigma^6)$ to $\omega(\sigma^4)$.

In~\cite{bendory2023autocorrelation} the authors consider a model of a sparse molecular structure 
where an atom is specified by a weighted Dirac
delta function, and a molecule is a sum $x = \sum_{i = 1}^m w_i \delta_{v_i}$ of such point masses. 
There, it is proved that the magnitudes of the support vectors are distinct and any two vectors are linearly independent then 
the second moment is sufficient to recover the $\O(3)$-orbit of the molecule. 
Our results imply that for the sparse molecular structures modeled in~\cite{bendory2023autocorrelation}
the $\O(3)$-orbit of the structure can be recovered without assuming that the support vectors have distinct magnitudes.
In addition we give an algorithm for recovery in this case. 
In cryo-EM other second moment recovery results were obtained for structures which are sparse with respect to a generic basis~\cite{bendory2024sample} and those which lie in a generic semi-algebraic set~\cite{bendory2025transversality}.

\subsection{Connection to Euclidean distance geometry} 
Euclidean distance geometry studies the recoverability of point configurations in $\r^n$ from a combination of (a) partial, (b) noisy, or (c) unlabeled interpoint distance data. We refer the reader to \cite{dokmanic2015euclidean} for a survey. 
Much research in Euclidean distance geometry has focused on the classical beltway problem on $S^1$ and the related turnpike problem of recovering points on the real line - see \cite{huang2021reconstructing, connelly2024reconstruction} and references therein. In the more general setting of unlabeled distance geometry in $\r^n$, a fundamental result due to Boutin and Kemper~\cite[Theorem~1.6]{boutin2003reconstructing} establishes that, when $m \leq 3$ or $m \geq n+2$, there exists a hypersurface $H \subset (\mathbb{R}^n)^m$ such that every point configuration $X = \{v_1, \dots, v_m\}$ outside $H$ can be reconstructed from its unlabeled interpoint distances, up to rigid transformations. In their subsequent work \cite{boutin2007which}, the authors also give a polynomial-time evaluation criterion for determining whether a given point configuration is reconstructible from its unlabeled distance data. A related generalization of their result is explored within theory of rigid graphs, where only a subset of unlabeled distances are assumed to be available~\cite{gortler2019generic, gkioulekas2024trilateration}.

Relating these results to the orthogonal beltway problem, we observe in Section~\ref{sec.distancegeometry} that if a signal $x$ is supported at $\{v_1, \dots, v_m\}\subset \r^n$, then the data contained in its second moment  determines the interpoint
distances $\|v_i-v_j\|$ of the support vectors. In general the second moment typically contains more information than just a set of unlabeled interpoint distances. In fact, it is shown in \cite{bendory2024beltway} that for generic point configurations in $\r^n$, a case in which each support vector has distinct norm almost surely, the second moment already contains a complete labeling of the interpoint distances and orbit recovery is easy. The interesting and novel case occurs when multiple support vectors have the same norm and the second moment only gives 
a partial labeling of the interpoint distances. In the extreme case, when all the points lie on a sphere then the second moment gives no labeling information and is consequently
equivalent to the set
of unlabeled interpoint distances. Our Corollary~\ref{cor.interpoint} then follows from
Theorem~\ref{thm:mainthm}.

Turning to applications, the \emph{unlabeled} distance geometry problem enjoys several applications including reconstruction of room geometries from first order echoes, echo-based SLAM  \cite{dokmanic2013echoes, dokmanic2014localize, boutin2020drone} and geopositioning \cite{boutin2024global, boutin2025developments}. These applications motivate the development of efficient algorithms for recovering point configurations from their unlabeled distance data. There do exist polynomial-time reconstruction algorithms, such as the one proposed in \cite{duxbury2016unassigned}, that guarantee recovery of generic configurations from unlabeled distances in $\r^n$. However, for the same reasons as before, it is unclear if this algorithm can readily be applied to configurations where multiple support vectors have the same norm. This observation leads to the natural question: can one recover the support of a generic signal $x$ from its \textit{second moment} assuming that at least $\ell$ of these vectors have the same norm? Since the second moment contains more information than the unlabeled set of interpoint distances, it is also natural to wonder if the algorithm in \cite{duxbury2016unassigned} can be further optimized when adapted to this setting. The other natural question is to ask if it is possible to perform recovery assuming that the second moment data is corrupted with low levels of noise. To our knowledge, the noisy variant of the unlabeled distance geometry problem remains largely unresolved, with no efficient algorithms currently available. We believe that studying recovery from noisy second moment data provides a promising starting point for addressing this broader challenge.

\section{Background}\label{sec:background}
We recall some background material from~\cite{bendory2024beltway}. A {\em signal} in $\r^n$ is a function $x:\r^n \to \r$. The orthogonal group $\O(n)$ acts naturally on a the space of
signals via 
the action $g \cdot x (v) = x (g^{-1}v)$. Equipping $\O(n)$ with the Haar measure, we define the second moment or auto-correlation, $m_2(x)$, of a signal $x$ as the function  $(\r^{n})^2 \to \r$ defined by the formula
$$m_2(x)(\tau_1, \tau_2) = \int_{\O(n)} (g\cdot x)(\tau_1) (g\cdot x)(\tau_2) dg.$$
We notice immediately that $m_2$ is $\O(n)-$invariant; i.e $m_2(x) = m_2(gx)$. Consequently any signal $x \in \r^n$ can be at best recovered up to its $\O(n)$-orbit from its second moment measurements $m_2(x)$.

It is also possible to define the auto-correlation of a compactly supported distribution, such as a Dirac delta function, thereby producing a distribution supported on $\O(n)$-invariant (compact) subsets of $(\r^n)^2$.  Given a finite collection of points $S= \{v_1,\dots, v_m\} \subset \r^n$ we can define a sparse signal $x$ supported on $S$ as 
\begin{equation}\label{eq:deltafunc}
x = \sum_{k=1}^m w_k\delta_{v_k}    
\end{equation}
where $\delta_{v}$ is the Dirac $\delta-$distribution denoting a point mass located at $v \in \r^n$ and
the $w_k$ are weights. Through a slight abuse of terminology, we will refer to such signals $x$ as $\delta$-functions
and say that $\delta$-function is {\em binary} if $w_k = 1$ for all $k$. There is a natural $\O(n)$ action on the space of $\delta$-functions given by $$g\cdot x = \sum_{k=1}^m w_k\delta_{g\cdot v_k}$$
In this setting, the second moment $m_2(x)$ is a distribution on $(\r^n)^2$ supported on the $\O(n)$-orbits $\o{(v_i, v_j)} = \{(g\cdot v_i, g\cdot v_j) | g\in G\}$ given by the formula 
\begin{equation}\label{eq:secondmoment}
m_2(x)(\tau_1, \tau_2) = \sum_{i,j = 1}^m  w_iw_j\mu_{v_iv_j}(\tau_1, \tau_2)    
\end{equation}
where $$\mu_{t_it_j}(\tau_1, \tau_2) = \int_{\O(n)} (g\cdot \delta_{v_i})(\tau_1)(g\cdot \delta_{v_j})(\tau_2) dg.$$
The beltway problem over an orthogonal group $\O(n)$ seeks to recover the $\O(n)$-orbit of a $\delta$-function $x \in \r^n$ from its second moment measurements \eqref{eq:secondmoment}.

\begin{Definition}{\cite[Definition 3.1]{bendory2024beltway}} \label{def:collisionfree}
    Let $S = \{v_1,\dots, v_m\} \subseteq \r^n$ be a set of points.
    \begin{enumerate}
        \item[a.] We say that the set $S$ is collision-free if for every pair of two-element subsets $\{v_i, v_j\}$ and $\{v_\ell, v_m\}$ we have $g\cdot \{v_i, v_j\} = \{v_\ell, v_m\}$ for some $g \in \O(n)$ if and only if $\{v_i, v_j\} =  \{v_\ell, v_m\}$.
        \item[b.] We say that the set $S$ is radially-collision-free if for every $v_i, v_j \in S$, we have $v_i = g\cdot v_j$ for some $g \in \O(n)$ if and only if $v_i = v_j$.
    \end{enumerate}
\end{Definition}
Given a binary $\delta$-function $x$ supported at $m \geq 3$ collision-free points in $\r^n$, \cite[Theorem 3.2]{bendory2024beltway} provides an explicit upper bound on the number of $\O(n)$-orbits of $\delta$-functions $y$ with the same second moment as $x$. 
When the support is radially collision-free the bound is 1, meaning that the $\O(n)$-orbit of a binary $\delta$-function with radially collision-free support is determined from its second moment.\\

\paragraph{{\bf The second moment and Gram matrices.}} Suppose $x = \sum_{i =1}^m \delta_{v_i}$ is a binary $\delta$-function. Let
$X= [v_1 \ldots v_m]$ be the $n \times m$ matrix whose columns are the support vectors taken in any order 
and let $A = X^TX$ be the corresponding Gram matrix.
\begin{Proposition}{\cite[Proposition 4.1]{bendory2024beltway}}\label{prop:triples}
If $x$ has a collision-free support then $m_2(x)$ determines the set of triples
$\mathcal{T}(X) = \{(A_{ii}, A_{jj}, A_{ij})\}_{i < j}.$
\end{Proposition}
In other words, the second moment of a binary $\delta$-function $x \in \r^n$ supported on
a set of collision free points $X= \{v_1, \ldots , v_m\}$ is the set of triples $\t(X^\top X) = \{(\|v_i\|^2, \|v_j\|^2, \ip{v_i}{v_j})\}$. It is crucial to note here that if all the $\|v_i\|$ are distinct, then $x$ has radially collision free support, and we obtain unique recovery immediately~\cite{bendory2024beltway}. Thus the second moment recovery problem is non-trivial only when 
some of the vectors in the support have the same norm. 

For binary $\delta$-functions supported on points on the unit sphere, \cite[Proposition 5.2]{bendory2024beltway} states that if $m \leq n$ there is a Zariski dense subset $\W$ such that for each $\delta$-function $x \in \W$, there are at least $(m-1)!$ non-orthogonally equivalent $\delta-$functions $y$ such that $m_2(x) = m_2(y)$. 
By contrast, our Theorem~\ref{thm:mainthm} implies that when $m>n$, the second moment determines the $\O(n)$-orbit
of a binary $\delta$-function supported at $m$ generic points on the sphere. In particular, this proves~\cite[Conjecture 5.4]{bendory2024beltway}. 

\section{Proof of Theorem \ref{thm:mainthm} and Corollary~\ref{cor.interpoint}}

\subsection{Preliminary notation.} 

\noindent{\bf Equivalent psd matrices.} Given an $m$-tuple of vectors $(v_1, \ldots  , v_m)\in (\r^{n})^m$ we use the notation $[v_1 \ldots v_m]$ to denote
the $n \times m$ matrix whose columns are the vectors $v_i$. If $X = [v_1 \ldots v_m]$ then the Gram matrix
$A=X^\top X$ is an $m \times m$ psd matrix of rank $l \leq \min\{m,n\}$. An $m \times m$ rank-$l$ psd matrix $B$ is said to be {\em equivalent} to $A$ if $B$ factors as $B = Y^TY$ where $Y$ is a $n \times m$ matrix obtained by permuting the columns of $X$.

\begin{Definition}\label{def.symmrearrangement}
   If $A= (A_{ij})$ and $B=(B_{ij})$ are symmetric $m \times m$ matrices then we say that $B$ is a 
{\em symmetric rearrangement} of $A$ if the following conditions are satisfied:
\begin{enumerate}
    \item $A$ and $B$ have the same diagonal.
    \item The set of triples $\mathcal{T}(A) = \{(A_{ii}, A_{jj}, A_{ij})\}$ equals the set of triples
    $\mathcal{T}(B) = \{(B_{ii}, B_{jj}, B_{ij})\}$.
\end{enumerate}
\end{Definition}

If  the diagonal entries of a symmetric matrix are identical, then any permutation of the off-diagonal entries that preserves the symmetry of
the matrix is a symmetric rearrangement. However, if the diagonal is non-constant then second condition of Definition~\ref{def.symmrearrangement} restricts the possible permutations of the off-diagonal entries. For 
example if $A = \begin{bmatrix} 1 & 2 & 3\\ 2 & 1 & 4 \\ 3 & 4 & 2\end{bmatrix}$,  
the only non-trivial symmetric rearrangement of $A$ is $B = \begin{bmatrix} 1 & 2 & 4\\2 & 1 & 3\\ 4 & 3 & 2\end{bmatrix}$.
The symmetric matrix $C = \begin{bmatrix} 1 & 3 & 2\\ 3 & 1 & 4\\2 & 4 & 2\end{bmatrix}$ is not a symmetric rearrangement
of $A$ because $\mathcal{T}(C)$ contains the triple $(C_{11}, C_{33}, C_{13}) = (1,2,2)$ which is not in the set
$\mathcal{T}(A)$. Indeed if the diagonal entries of a symmetric matrix are all distinct then it has no non-trivial symmetric rearrangements.

Note that if $A$ and $B$ are equivalent psd matrices with the same diagonal then $B$ is a symmetric rearrangement of $A$. However the converse does not hold. In general a symmetric rearrangement of a low-rank psd matrix need not not have low rank nor be psd. Indeed the strategy of the proof Theorem~\ref{thm:mainthm} is to show that if $A$ is a `generic' low rank psd matrix with specified
diagonal configuration then the only symmetric rearrangements of the same rank are equivalent to $A$.
\begin{Example}
Consider the rank $3$ matrix 
$$X = [v_1\, \dots\, v_4] = \begin{bmatrix}
    1/\sqrt{6} & -2/\sqrt{14} & 3/\sqrt{26} & 5/\sqrt{50} \\
    2/\sqrt{6} & 1/\sqrt{14} & 1/\sqrt{26} & 4/\sqrt{50} \\
    1/\sqrt{6} & 3/\sqrt{14} & 4/\sqrt{26} & -3/\sqrt{50}
\end{bmatrix}$$representing a collection of $4$ points on the unit sphere in $\r^3$.
Let $$A = X^\top X = \begin{bmatrix}
 1. & 0.356348 & 0.720577 & 0.57735 \\
 0.356348 & 1. & 0.470757 & -0.555492 \\
 0.720577 & 0.470757 & 1. & 0.194145 \\
 0.57735 & -0.555492 & 0.194145 & 1. \\
\end{bmatrix}$$
be the Gram matrix of $X$. It has rank 3 and its eigenvalues are $\{2.1077,1.60268,0.289624,0\}$.
The matrix 
$$B = \begin{bmatrix}
 1. & 0.720577 & 0.470757 & 0.57735 \\
 0.720577 & 1. & -0.555492 & 0.194145 \\
 0.470757 & -0.555492 & 1. & 0.356348 \\
 0.57735 & 0.194145 & 0.356348 & 1. \\
\end{bmatrix}
$$
is a symmetric rearrangement which has rank 4 and eigenvalues $\{2.07115,1.58357,0.51753,-0.172247\}$.
On the other hand the symmetric rearrangement 
$$C = \begin{bmatrix}
 1. & 0.470757 & -0.555492 & 0.356348 \\
 0.470757 & 1. & 0.194145 & 0.720577 \\
 -0.555492 & 0.194145 & 1. & 0.57735 \\
 0.356348 & 0.720577 & 0.57735 & 1. \\
\end{bmatrix}
$$
has the same eigenvalues as $A$, and is equivalent to $A$ because it is the Gram matrix of the matrix $$Y = [v_2 v_3 v_4 v_1] = \begin{bmatrix}
    -2/\sqrt{14} & 3/\sqrt{26} & 5/\sqrt{50}  & 1/\sqrt{6}\\
     1/\sqrt{14} & 1/\sqrt{26} & 4/\sqrt{50}  & 2/\sqrt{6}\\
   3/\sqrt{14} & 4/\sqrt{26} & -3/\sqrt{50} & 1/\sqrt{6}
\end{bmatrix} 
$$
It is possible to construct point configurations $X$ whose Gram matrix $A= X^\top X$ has a non-equivalent symmetric rearrangement which is a Gram matrix of the same rank. See \cite[Example 5.5]{bendory2024beltway}
\end{Example}
\subsection{Reduction to a statement about symmetric rearrangements of psd matrices.}
Our goal is to show that there is a non-empty Zariski open set $U^\ell$ in $\V^{\ell}_{m,n}$ such that
for all $x \in U^\ell$, $m_2(x)$ determines the $\O(n)$-orbit of $x$.
To do this, we may assume that the points in the support of $x$ span all of $\r^n$ and are collision free,
since the locus of $\delta$-functions with these properties is a dense Zariski open set in $\V^\ell_{m,n}$ for every $\ell$.

Given a set of points $S =\{v_1, \ldots , v_m\} \subset \r^n$ with $m > n$, choose any ordering 
of the points and let $A = X^\top X$ where $X = [v_1 \ldots v_m]$. 
In particular, if $x = \sum_{i=1}^m \delta_{v_i}$
then, as previously observed, $m_2(x)$ determines the set of triples $\{(A_{ii}, A_{jj}, A_{ij})\}$ for
all pairs $(i,j)$. 
Now suppose $T = \{w_1, \ldots ,w_m\}$ is another set of points and let $y = \sum_{i=1}^m \delta_{w_i}$. If $m_2(x) = m_2(y)$ then the sets of magnitudes (counted with multiplicity) $\{|v_1|, \ldots , |v_m|\}$
and $\{ |w_1|, \ldots , |w_m|\}$ are equal. Thus, after possibly reordering the points of $T$ we can assume that $|w_i| = |v_i|$ 
which implies that if $Y= [w_1 \ldots w_m]$ then $Y^TY$ is a symmetric rearrangement of $X^\top X$.

Hence, to prove Theorem~\ref{thm:mainthm} we must show that there is a non-empty Zariski open set
$U^\ell \subset \V^\ell_{m,n}$ such that if $x \in U^\ell$ and 
$m_2(x) = m_2(y)$ for $y \in \V^{\ell}_{m,n}$ then 
the corresponding Gram matrices associated to $x$ and $y$ are equivalent.
Let $\mathcal{P}_{n,m}$ be the set of rank-$n$ psd $m \times m$ matrices
and let $\mathcal{P}^\ell_{m,n}$ be the closed subset defined by the condition that 
that a set of at least $\ell$ entries on the diagonal are equal. 
To prove Theorem~\ref{thm:mainthm} we will prove the following
statement about Gram matrices. 

\begin{Proposition} \label{prop.rearrangement}
For $0 \leq \ell \leq m$ there exists a dense Zariski open set in $U^{\ell}_{m,n} \subset \mathcal{P}_{n,m}^{\ell}$
such that for every $A \in U^{\ell}_{m,n}$ any non-equivalent symmetric rearrangement of $A$ 
has rank more than $n$.
\end{Proposition}

\subsection{Reduction to the case $\ell = m$}\label{sec:relaxationtom-1}
To prove the proposition it suffices to show that for each $\ell$ the set of matrices in $A \in \mathcal{P}^{\ell}_{m,n}$
which do not satisfy the conclusion of Proposition~\ref{prop.rearrangement} lie in a Zariski closed subset of 
$\mathcal{P}^{\ell}_{m,n}$ of strictly smaller dimension. Let $Z \subset \mathcal{P}_{n,m}$ parametrize matrices $A$ that have non-equivalent rank $n$ psd rearrangements. To prove the proposition we need to show that 
the open set $U =\mathcal{P}_{n,m} \setminus\overline{Z}$ has dense intersection with $\mathcal{P}^{\ell}_{n,m}$
for every $\ell$.
Since psd matrices with all diagonal entries distinct 
satisfy the conclusion of Proposition~\ref{prop.rearrangement} by~\cite[Theorem 3.2]{bendory2024beltway} we know
that $U$ is non-empty. Moreover, if $\ell' > \ell$ then $\mathcal{P}^{\ell'}_{n,m} \subset \P^{\ell}_{n,m}$.
Thus, it suffices to show that the intersection of $U$ with $\mathcal{P}^{m}_{m,n}$ is Zariski dense.
In other words, it suffices to prove the following proposition.
\begin{Proposition} \label{prop.pmn}
    There is a Zariski open set $U^{m}_{m,n} \subset \mathcal{P}^{m}_{m,n}$ such that for all $A \in U^{m}$ any non-equivalent symmetric rearrangement of $A$ is not in $\mathcal{P}^{m}_{m,n}$.
\end{Proposition}
\subsubsection{Proof of Proposition~\ref{prop.pmn}}
Proposition~\ref{prop.pmn} is an immediate consequence of the following two lemmas.

\begin{Lemma}\label{lem:Smnirred}
If $n \geq 2$. The space $\mathcal{S}_{m,n}$ parameterizing $m\times m$ complex symmetric matrices $A$  of rank $\leq n$ 
such that $A_{11} = A_{22} = \ldots  = A_{m,m}$
is an irreducible complex variety.
\end{Lemma}

\begin{proof}
By Lemma~\ref{lem:takagi} in Appendix~\ref{sec:Takagi}, the locus of complex symmetric matrices with rank $\leq n$ is the image of $\c^{n\times m}$ under the map $\phi\colon \c^{n \times m} \to \c^{m \times m}; X \mapsto X^\top X$. Since the image of an irreducible set is irreducible, it suffices to find an irreducible subvariety of $\c^{n\times m}$ whose image under $\phi$ is $\mathcal{S}_{m,n}$. 
Let
$$V = \{(v_1,\dots, v_{m}, s) \, | \, f(v_1, s) = f(v_2,s) = \dots = f(v_{m},s) = 0\} \subseteq \c^{n\times m} \times \mathbb{A}^1$$
where $$f(x, s) = s-\sum_{j=1}^n x_{j}^2 = s-x^\top x.$$

If $W$ is the projection of $V$ to $\c^{n \times m}$ then $\phi(W) = \mathcal{S}_{n,m}$.
Thus it suffices to prove that $V$ is irreducible. To see this note that $V$ is isomorphic $m$-fold fiber product $H \times_{\mathbb{A}^1} H \times_{\mathbb{A}^1} \ldots \times_{\mathbb{A}^1} H$
where $H \subset \c^n \times \mathbb{A}^1$ is the hypersurface defined the equation $f(x,s) = s - \sum_{j=1}^n x_j^2$. Since $n \geq 2$ this hypersurface is irreducible and non-singular so the projection $H \to \mathbb{A}^1$ is flat. Hence $V \to \mathbb{A}^1$ is also flat, and the generic fiber is irreducible, so $V$ is irreducible. (Note that
if $n > 2$ then all fibers of $V \to \mathbb{A}^1$ are irreducible.)
\end{proof}

\begin{Lemma} \label{prop.relaxation}
There is a non-empty Zariski open subset of $\mathcal{S}_{m,n}$ such that for all $A$ is in this open set
if $B$ is a 
non-equivalent symmetric rearrangement of $A$ then $\rank{B} > n$. 
\end{Lemma}
\begin{proof}
Let $A$
be an arbitrary matrix in $\mathcal{S}_{m,n}$. Since the diagonal of $A$ is constant, we
 know by Lemma~\ref{lem:takagi} that there are column vectors $v_1 \ldots , v_m \in \c^n$ such
 that $A_{ij} =v_i^T v_j$ for $i < j$,  and $A_{ii} = v_i^T v_i=s$. 
 
Since $A$ has rank at most $n$ the $(n+1) \times (n+1)$ 
minors of $A$ must vanish. These minors are homogeneous polynomials of degree $n+1$ in the variables
$\{s, A_{ij}\}_{i,j \in [m-1], \, i<j}$.
Let $B$ be a symmetric rearrangement 
of $A$. If $B$ also has rank at most $n$ then the $(n+1)\times (n+1)$ minors of $B$, which are homogeneous polynomials
of degree $n+1$
in the same set of variables, must also vanish. 

From Lemma~\ref{lem:Smnirred} we know that $\mathcal{S}_{m,n}$ is irreducible. We will show that for any non-equivalent symmetric rearrangement $B$ of $A$, there exists a minor $f$ of $B$ that is not a linear combination of the $(n+1)\times (n+1)$ minors of $A$. By the homogeneity of the minors involved, this would imply that $f$ is not in the ideal $I$ generated by the $(n+1)\times (n+1)$ minors of $A$. Thus the set of matrices $A
\in S_{m,n}$ such that there exists a non-equivalent symmetric rearrangement  $B$
must 
lie in a Zariski closed subset of strictly smaller dimension than $\dim S_{m,n}$.

If $m =3$ (in which case we must have $n=2$) all symmetric rearrangements of $A$ are equivalent, as observed
 in the proof of~\cite[Lemma 1.2]{boutin2007which}, so the proposition is trivially true.

 Now assume $m \geq 5$ and let $B$ be a non-equivalent symmetric rearrangement of $A$. 
 Then there is a permutation $\phi$ of the set of distinct pairs $\{i,j\}$ such
 $B_{pq} = A_{\phi\{ij\}}$. Since $B$ is not equivalent to $A$ by~\cite[Lemma 1.2]{boutin2007which}
 there exists a distinct pair of indices $(i,j),(i,k)$ such that $\phi\{i,j\} \cap \phi\{i,k\} = \emptyset$.
 This means that in matrix $B$ the entries $A_{ij}= A_{ji}$, $A_{ik}=A_{ki}$ lie in different rows and columns. In particular
 this implies that $B$ has an $(n+1) \times (n+1)$ minor\footnote{Note that we assume that $n \geq 2$ so $n+1 \geq 3$.} which contains a monomial divisible by
 $A_{ij} A_{ji} A_{ik} = A_{ij}^2 A_{ik}$. On the other hand, no such monomial can appear in an $(n+1) \times (n+1)$ minor of $A$
since  in the matrix $A$, $A_{ik}$ is in the same row as $A_{ij}$ and $A_{ki} = A_{ik}$ is in the same column
as $A_{ji} = A_{ij}$.

If $m =4$ the criterion of~\cite[Lemma 1.2]{boutin2007which} implies that if $B$ is a non-equivalent symmetric rearrangement
of $A$ then $\phi\{1,2\} \cap \phi\{1,3\} \cap \phi\{1,4\} = \emptyset$ meaning that
in the matrix $B$ the entries $A_{12}, A_{13}, A_{14}$ do not lie in a single column or row, and once again
we can find an $(n+1) \times (n+1)$ minor of $B$ which cannot be expressed as a linear combination of the minors of $A$ as illustrated
in Example~\ref{ex.34} below.
\end{proof}
\subsection{Examples} \,\\

%\arun{[[Erased the words ``the matrix" to fix alignment issue with the variables.]]}
\begin{Example}\label{ex.34}
Consider $A = \begin{bmatrix} s & A_{12} & A_{13} & A_{14}\\
A_{12} & s & A_{23} & A_{24}\\
A_{13} & A_{23} & s & A_{34}\\
A_{14} & A_{24} & A_{34} & s\end{bmatrix}$ in the 7 indeterminates~$s, A_{12}, A_{13}, A_{14}, A_{23}, A_{24}, A_{34}$
and let $B$ be the symmetric rearrangement 
$B = \begin{bmatrix} s & A_{23} & A_{13}&  A_{14}\\
A_{23} & s & A_{12} & A_{24}\\
A_{13} & A_{12} & s & A_{34}\\
A_{14} & A_{24} & A_{34} & s\end{bmatrix}$
which corresponds to a permutation $\phi$ with $\phi\{1,2\} = \{2,3\}, \phi\{1,3\} = \{1,3\}, \phi\{1,4\} = \{1,4\}$.
In the matrix $B$ both instances of $A_{14}$ and $A_{12}$ lie in different rows and columns so $B$ has a minor which contains a monomial divisible by $A_{14}^2 A_{12}$. However, no minor of $A$ contains
such a monomial because $A_{12}$ and one instance of $A_{14}$ lie in the same row.
\end{Example}
\begin{Example}
%    Drawing connections to Convex geometry, the well-known Schoenberg criterion \cite{Schoenberg1935RemarksTM} (when adapted to our context) states that a symmetric matrix $A \in \r^{m\times m}$ is the Gram matrix associated to some $X \in \r^{n\times m}$ if and only if $A$ is positive-semidefinite and $\rank{A} \leq n$. 
%\dan{I think that this is a bit of reach to connect it with the Schoenberg criterion.}
Our main result states that if $A$ is a Gram matrix associated to a generic 
$x \in \mathcal{V}_{m,n}^{\ell}$, then any non-equivalent symmetric rearrangement of $A$ is not the Gram matrix of a point configuration in $\r^{n}$, since the rearranged matrix has rank more than $n$. However, our result does not preclude the possibility that a non-equivalent symmetric rearrangement of $A$ is psd and thus the Gram matrix associated to a
configuration of $m$ points in $\r^{k}$ with $k > n$.
Indeed, consider
$$
X =
\begin{bmatrix}
 0.5673 & -0.4593 &  0.1548 &  1.0000 \\
 0.4515 & -0.4885 & -0.0278 &  0.0000 \\
-0.6887 & -0.7419 &  0.9876 &  0.0000
\end{bmatrix}
\in \mathbb{R}^{3\times 4}.
$$

Its Gram matrix is given by
$$
A = \begin{bmatrix}
 1.0000 &  0.0299 & -0.6049 &  0.5673 \\
 0.0299 &  1.0000 & -0.7902 & -0.4593 \\
-0.6049 & -0.7902 &  1.0000 &  0.1548 \\
 0.5673 & -0.4593 &  0.1548 &  1.0000
\end{bmatrix} \in \r^{4\times 4}
$$

Now consider the symmetric rearrangement defined by the transposition $A_{12} \leftrightarrow A_{13}$. This produces the rearranged matrix
$$
B =
\begin{bmatrix}
 1.0000 & -0.6049 &  0.0299 &  0.5673 \\
-0.6049 &  1.0000 & -0.7902 & -0.4593 \\
 0.0299 & -0.7902 &  1.0000 &  0.1548 \\
 0.5673 & -0.4593 &  0.1548 &  1.0000
\end{bmatrix}
$$
which has eigenvalues 
$\lambda(B) = (
2.3522, 1.1582, 0.0191, 0.4705)$. Thus $B$ is positive definite
matrix of rank $4$ which is the Gram matrix of the point configuration
$$X' = \begin{bmatrix}
1.0000 & -0.6049 &  0.0299 &  0.5673 \\
0.0000 &  0.7963 & -0.9696 & -0.1458 \\
0.0000 &  0.0000 &  0.2428 & -0.0147 \\
0.0000 &  0.0000 &  0.0000 &  0.8104
\end{bmatrix} \in \mathbb{R}^{4\times 4}.$$

Computational experiments indicate that the existence of psd symmetric rearrangements of higher rank is a common phenomenon. We initialized $1000$ uniformly random $3\times 4$ matrices $X$ such that all the columns of $X$ had norm $1$. In this setup, each Gram matrix $X^\top X$ admits a total of $\binom{4}{2}! = 720$ symmetric rearrangements with exactly $720 - 4! = 696$ of them not equivalent to $X^\top X$. For each Gram matrix $X^\top X$, we iterated through all of its $696$ non-equivalent symmetric rearrangements, and observed the following. After $1000$ trials, we saw that $786$ ($78.6\%$) of the rank $3$ psd matrices $X^\top X$ admitted at least one symmetric rearrangement that was positive semidefinite of higher rank (equal to $4$). Across the entire experiment we observed a total of $5384$ non-equivalent symmetric rearrangements that were positive semidefinite, representing approximately $1\%$ of the possible non-equivalent symmetric rearrangements of any given Gram matrix $X^\top X$. This suggests that if $n < m$ there exists a set of positive measure in the variety $\mathcal{P}_{m,n}$ such that each matrix in this set has a symmetric rearrangement which is psd of higher rank $> n$. \\

Repeating this experiment for $5$ points on a unit sphere quickly gets out of hand, since in this case, each Gram matrix has $10!$ possible symmetric rearrangements. We can, however, scale this experiment to five points, if we assume that one of the supports has a distinct norm. To this end, we initialized $1000$ uniformly random $3 \times 5$ matrices $X$ such that the first four columns of $X$ had norm $1$ and the last column had norm $2$. In such a setup, the Gram matrix $X^\top X$ admits a total of $\binom{4}{2}! - 1 = 6! - 1 = 719$ non-equivalent symmetric rearrangements. The reason for this will be made clear in the Section~\ref{sec:recoveringx}. For each Gram matrix $X^\top X$, we iterated through all of its $719$ non-equivalent symmetric rearrangements, and observed the following. After $1000$ trials, we saw that $443$ ($44.3\%$) of the rank $3$ psd matrices admitted at least one symmetric rearrangement that was positive semidefinite of higher rank. Across the entire experiment we observed a total of $6536$ non-equivalent symmetric rearrangements that were positive semidefinite representing approximately $2\%$ of the possible non-equivalent symmetric rearrangements.
\end{Example}
\subsection{Proof of Corollary~\ref{cor.interpoint} and connection to the work 
of Boutin and Kemper} \label{sec.distancegeometry}
Corollary~\ref{cor.interpoint} follows immediately from Theorem~\ref{thm:mainthm} and the following lemma.
\begin{Lemma} \label{lem.interpoint}
If $x = \sum_{i=1}^m \delta_{v_i}$ then $m_2(x)$ determines the set of interpoint distances 
in the set $S= \{v_1, \ldots , v_d\}$.
If all of the $v_i$ have the same magnitude then  the second moment $m_2(x)$ is equivalent to the set of (unlabeled) interpoint distances.
\end{Lemma}
\begin{proof}
By Proposition~\ref{prop:triples} the second moment of the binary delta function $x = \sum_{i=1}^m \delta_{v_i}$
is the set of triples $\mathcal{T}(X) = \{(\|v_i\|^2, \|v_j\|^2, \langle v_i, v_j \rangle)\}$.
From this data we
determine the set of interpoint distances  $\{d_{ij} = \|v_i - v_j\|\}$ 
via the equation $d_{ij}^2 = \|v_i\|^2 + \|v_j\|^2 - 2\ip{v_i}{v_j}$. 
If all of the $v_i$ have the same magnitude then we have no labeling information of the triples
$(\|v_i\|^2, \|v_j\|^2, \langle v_i, v_j \rangle)$ and therefore no labeling
information on the interpoint distances $d_{ij}$.
\end{proof}
\begin{Remark}
If the magnitudes of the vectors
are distinct (i.e. the support of $x$ is radially collision free) then we can sort
the vectors by magnitude and obtain a complete labeling of
the interpoint distances $d_{ij}$ from the second moment, and the recovery of the $\O(n)$-orbit from this data is immediate~\cite{bendory2024beltway}.
When there are support vectors with equal magnitude, the second moment provides a partial labeling
of these distances, and when all the $v_i$ have equal magnitude, then the second moment provides
no labeling. 
\end{Remark}
%$\ell =m$ then $\V^{\ell}_{m,n}$ consists of delta-functions supported at $m$ points $v_1, \ldots , v_m$ on the sphere. 
%In this case we observe that the relation $d_{ij}^2 = \|v_i\|^2 + \|v_j\|^2 - 2\ip{v_i}{v_j}$
%implies that the second moment gives exactly the set of unlabeled interpoint distances since the magnitudes of the vectors
%$v_i$ all have the same value - for example we can normalize and assume $\|v_i\| = 1$ for all $i$.
%This implies the following corollary.
%\arun{[[Do we need to change the presentation of the paragraph environment?]]}
\paragraph{{\bf Connection to the work of Boutin and Kemper}}
In~\cite{boutin2003reconstructing} Boutin and Kemper considered the problem of recovering
    $m$ points in $\r^n$ from their unlabeled interpoint distances.
    They proved that when $m \geq n+2$ 
    there exists a hypersurface in $\r^{n\times m}$ such that any support configuration from the complement of this hypersurface can be recovered up to rigid motions from its unlabeled set of interpoint distances. 
    %Since any collection of
    %vectors in $\r^n$ can be centered at the origin, their result is equivalent to the statement that a generic collection
    %of points centered at the origin can be recovered up to orthogonal transformation from their interpoint distances.
    Our Corollary~\ref{cor.interpoint} cannot be deduced from the result of Boutin and Kemper because there is no way to know if the sphere is contained in the complement of their hypersurface. In addition, we improve their bound from $m \geq n+2$ to $m \geq n+1$ but we note that our result does not preclude the possibility (at least with $m = n+1$) that
    there exists a Zariski dense set of configurations of $m$ points on $S^{n-1}$ such that for each 
    configuration $S$ in this set we can find a configuration of $m$ points in $\r^n$ which do not all lie on a common sphere but have the same set of interpoint distances as $S$.

\section{Algorithms for orbit recovery from the second moment}\label{sec:recoveringx}
Let $x \in \V_{m,n}^\ell$ be a binary $\delta$-function supported on a collision-free set $S$ with at least $\ell$ distinct support vectors having the same norm. In this section we seek to recover the $\O(n)$-orbit of $S$ of the $\delta$-function $x$ from knowledge of its second moment $m_2(X)$. 
As observed in Section~\ref{sec:background}, this is equivalent to recovering the equivalence class of the Gram matrix $A = X^TX$ 
from the set of triples $\t(A)$.
We begin by observing that since the $m\times m$ symmetric matrix $A$ has rank $\leq n$, a naive approach to determine $A$ from $\t(A)$ is to construct every possible symmetric matrix $B$ from $\t(A)$ and check if $\mathrm{rank}(B) \leq n$. For generic $S$, this procedure eventually recovers 
a matrix equivalent to $A$, since Theorem~\ref{thm:mainthm} guarantees that there is a unique way up to equivalence to assemble the entries of $\t(A)$ into an $m\times m$ symmetric matrix that has rank at most $n$. 
%Upon obtaining the correct $A$, it can be readily decomposed as $X^\top X$ to obtain $X\in \r^{n\times m}$ corresponding to some ordering of the support vectors $S$ (up-to orthogonal translates).  
However, the computational cost of this procedure is prohibitive because the number of possibilities grows as $O(m^2!)$, making it impractical even for modest values of $m$. This motivates the search for algorithms that reduce the computational complexity.

To this end, in Section~\ref{sec:noiseless} we show that if the support $S$ is assumed to contain at least one vector of distinct magnitude, then the information contained in the second moment is equivalent to the knowledge of \textit{partially labeled} interpoint distances between vectors in $S$. Given this understanding, we reformulate our recovery problem as a relaxation of the unlabeled distance geometry problem and develop an algorithm to recover such support configurations given that the distances measured are noiseless. Since the second-moment $\t(A)$ in this setting contains strictly more information than just an unlabeled set of interpoint distances, we demonstrate that our algorithm offers a meaningful improvement to the time-complexity of unlabeled reconstruction algorithms such as the one in \cite{duxbury2016unassigned}. Subsequently, in Section~\ref{sec:unitsphere}, we consider the case where $x\in \V_{m,n}^m$ and assume that all support vectors have the same norm. Recall that in this case our recovery problem is equivalent to the unlabeled distance geometry problem where the point configuration to be recovered is contained completely inside the unit sphere. We adapt our main algorithm developed in Section~\ref{sec:noiseless} to this setting and illustrate its successful performance while still retaining a slight edge over the time complexity of the algorithm in \cite{duxbury2016unassigned} (though this edge vanishes as $n\to \infty$). Finally, in Section~\ref{sec:noisyalg} we go back to the relaxed setting (with supports having at least one point with distinct norm) and extend our main algorithm to perform recovery even when the unlabeled distances are corrupted with low levels of noise. \\

%\arun{[[Another Paragraph environment]]}
\paragraph{{\bf Connection to the Tribond algorithm~\cite{duxbury2016unassigned}.}}
For a fixed Euclidean dimension $n$, \cite[Theorem 4]{duxbury2016unassigned} introduces a polynomial time algorithm called the \textit{Tribond algorithm} that performs reconstruction of generic $n$ dimensional point configurations from unlabeled interpoint distances, assuming that the distances known are exact. The Tribond algorithm runs in time complexity $O(m^{n(n+2)-2}\ln(m))$ and proceeds by first by identifying the smallest rigid structure -- an $(n+2)$-clique, or equivalently, two $n$-simplices sharing a common $(n-1)$-dimensional base, bridged by an edge connecting the two remaining free vertices. Once such a small core is found, the algorithm builds iteratively by searching for compatible sets of $n+1$ distances that attach an additional vertex to the existing core, one step at a time until all distances are exhausted. By proceeding in this manner, the algorithm assumes that there is a unique point configuration within $\r^n$ which produced the unlabeled distance measurements and that any sub-structure obtained by the algorithm is a part of the final unique structure. The authors, in their subsequent work \cite{Billinge2016Assigned} call this property \textit{strongly generic}. Accordingly, it is true that distance distributions generated by generic configurations in $\r^n$ satisfy these assumptions. A related generalization of this property, called Trilateration, is explored in \cite{gkioulekas2024trilateration} in the context of rigid graphs. 

There are, however, two major considerations with this approach as it applies to our setup. First, the success of this algorithm relies on the assumption that the underlying configuration is generic in $\mathbb{R}^n$. However, such a configuration is almost surely radially collision-free, and consequently, the second moment $m_2(x)$ already contains enough information to uniquely recover the Gram matrix \cite[Theorem 3.1]{bendory2024beltway}. Moreover, it is not clear if the assumptions underlying the Tribond algorithm apply to our particular setup where multiple support vectors have the same norm. In addition, the most pressing limitation is computational. The bottleneck lies in testing all combinations of pairs of compatible $n-$simplices such that there exists a final distance that serves as an edge bridging the non-base vertices. This over-constrained search leads to large redundancies in our setting. In our setup, we enjoy the benefits of partial labeling pertaining to distances to the isolated point with distinct norm. This added flexibility allows us to propose an algorithm that has a time complexity bounded from above by $O(m^{8})$ for point configurations in $\r^3$, and more generally $O(n^3m^{n(n-1)+2})$ for point configurations in $\r^n$. This is an improvement to the exponent $n^2+2n-2$ seen in \cite{duxbury2016unassigned}, and notably a considerable improvement in the crucial case when $n=3$. Furthermore, for generic $x\in \V^{m-1}_{m,n}$, empirical evidence suggests that the performance of the algorithm is substantially below this upper bound.

\subsection{Recovering $x$ from exact $m_2(x)$ measurements}\label{sec:noiseless}
We begin with some setup. Let $\ell \leq m-1$ and let $x\in U^\ell \subseteq \mathcal{V}^\ell_{m,n}$ be a $\delta$-function with support $S = \{v_1,\dots, v_m\}$. For robustness, we may shrink $U^\ell$ to ensure that the support vectors $S = \{v_1,\dots, v_m\} \subseteq \mathbb{R}^n$ satisfy the condition that any three of them are linearly independent. This is equivalent to saying that given any ordering of the vectors in $S$, 
if $X=[v_1 \ldots v_m]$ then every $n \times 3$ submatrix of $X$ has full rank. A relevant consequence of this assumption is that any three support vectors, together with the origin are guaranteed to constitute the vertices of a tetrahedron having positive volume. Moreover, since $\ell \leq m-1$, we may shrink $U^\ell$ further to assume, without loss of generality, that $v_m$ is an isolated point with norm distinct from the remaining support vectors.

As observed in the proof of Lemma~\ref{lem.interpoint}, that the second moment $\t(A)$ contains enough information to extract the interpoint distance between $v_i$ and $v_j$ via the equation $d_{ij}^2 = \|v_i\|^2 + \|v_j\|^2 -2\ip{v_i}{v_j}$. Thus, the data of $\t(A)$ can be reorganized as a product of sets $D= \prod_{(r_1, r_2)} D(r_1, r_2)$, where
$$
D(r_1,r_2) = \{d_{ij}^2 \mid \|v_i\| = r_1,\ \|v_j\| = r_2\},
$$
indexed over pairs $(r_1,r_2)$ of support norms. To reduce notational clutter, we perform a similar reduction as in Section~\ref{sec:relaxationtom-1}, and assume, without loss of generality, that only two such classes occur $(r_1, r_1)$ and $(r_1,r_2)$. By rescaling $d_{ij}^2$ as needed, we may assume that $r_1=1$ and $r_2=r$. Recall that $\P_{m,n}^\ell$ denotes the set of $m \times m$ symmetric rank-$n$ psd matrices with at-least $\ell$ diagonal entries equal. 
%With this notation in place, we establish the following useful Lemma. 

\begin{Lemma}\label{lem:XTXYTY}
    Let $A_1, A_2$ be two matrices in $\mathcal{P}^{m-1}_{n,m} \setminus \mathcal{P}^{m}_{n,m}$ such that $A_2$ is a symmetric rearrangement of
    $A_1$. Then $A_1, A_2$ are equivalent to matrices of the form
    \begin{equation} \label{eq.p1p2} A_1 = \begin{bmatrix}
        P_1 & \alpha \\
        \alpha^T & r
    \end{bmatrix} \qquad A_2 = \begin{bmatrix}
        P_2 & \alpha \\
        \alpha^T & r
    \end{bmatrix}
    \end{equation}
    where 
    $\alpha= (a_1,\ldots , a_{m-1})^T \in \r^{m-1}$, $P_1$ has constant diagonal, and the $(m-1) \times (m-1)$ block $P_2$ is a (psd) symmetric rearrangement of the block $P_1$.
\end{Lemma}
\begin{proof}
Since $A_1$ and $A_2$ are psd of rank $n$ we know that we can factor
$A_1 = X^\top X$ and $A_2 = Y^TY$ where $X = [v_1 \ldots v_m]$ and $Y = [w_1 \ldots v_m]$
for some vectors $v_1, \ldots , v_m, w_1, \ldots , w_m \in \r^n$. Since $A_1$ and $A_2$ have the same diagonals we know that after reordering the points to produce equivalent Gram matrices
we may assume that the diagonals of $A_1$ and $A_2$ are the same and the first $m-1$ diagonal entries
have constant value $s$. Since $A_1, A_2 \in \mathcal{P}^{m-1}_{n,m} \setminus \mathcal{P}^{m}_{m,n}$
we know that $(A_1)_{mm} = (A_2)_{mm}=r$ for some positive integer $r\neq s$. Moreover, 
since $A_2$ is a symmetric rearrangement of $A_1$ the set of triples $\{((A_1)_{ii}, (A_1)_{mm},(A_1)_{im}) =(s,r,(A_1)_{im})\}$
is the same as the set of triples $\{((A_2)_{ii}, (A_2)_{mm}, (A_2)_{im})= (s,r,(A_2)_{im})\}$. In particular
the first $m-1$ entries of the $m$-th column of $A_2$ are a permutation of the first
$m-1$ entries of the first column of $A_1$. The first $m-1$ entries of the $m$-th column of $A_1$ are the inner products
$\{\langle v_1, v_m \rangle, \ldots , \langle v_1, v_{m-1} \rangle \}$ while the the first $m-1$ entries
of the $m$-th column of $A_2$ are the inner products $\{\langle w_1, w_m \rangle, \ldots , \langle w_1, w_{m-1} \rangle\}$. So
after possibly reordering the points $w_1,\ldots , w_{m-1}$ which corresponds to replacing $A_2$ by an equivalent matrix we can assume that $(A_1)_{mi} = (A_2)_{mi}$ for $i =1 \ldots , m-1$. Note that this reordering does not affect the values along the diagonal. It follows that the matrices $A_1$ and $A_2$ have the form given in~\eqref{eq.p1p2}.
\end{proof}

This lemma, viewed in conjunction with Theorem~\ref{thm:mainthm} implies that upon extracting the values $\ip{v_i}{v_m}$ in $\t(A)$ and labeling them (in some order), there exists a unique way to arrange the remaining $\ip{v_i}{v_j}$ into an $m\times m$ matrix $A$ such that $\rank{A}\leq n$. Equivalently, by assigning labels to the $m-1$ interpoint squared distances $d_{im}^2$ in $D(1,r)$, there exists a unique labeling of $d_{ij}^2 \in D(1,1)$ that realizes them as interpoint distances of some point configuration in $S^{n-1}\subseteq \r^n$. Thus, at this stage, we ultimately have a partially labeled set of interpoint distances $D = D(1,1) \times D(1,r)$ where all the squared distances $d_{im}^2 \in D(1,r)$ are labeled, and our goal is reduced to labeling the squared distances in the set $D(1,1)$, or equivalently, the inner products $\ip{v_i}{v_j}$ with $i,j<m$.

\subsubsection{Pre-processing step}

%Our approach begins with a pre-processing step that selectively trims the search space. 
Suppose we are trying to recover
the Gram matrix associated to the support $S = \{v_1,\dots,v_m\}$ of a $\delta$-function $x \in U^{m-1}$ where $\|v_i\|=1$ for all $i<m$ and $\|v_m\|=r$,
and any three vectors are linearly independent.
%The pre-processing step is powered by the following observation: 
For any choice of $i,j<m$, the four points $T = \{0, v_i, v_j, v_m\}$ form a tetrahedron (of positive volume, by assumption) with side lengths \{$1$, $1$ ,$2$, $d_{jm}$, $d_{im}$ and $d_{ij}$\}. Since the distances in $D(1,r)$ are labeled, the values of $d_{im}$ and $d_{im}$ are known, and only the distance $d_{ij}$ remains unknown. The following crucial result
originally proved by Menger~\cite{menger1928untersuchungen} establishes the conditions required for six positive numbers to form the edge lengths of a tetrahedron. We use this result to reduce the possible values of $d_{ij}$ among the unlabeled interpoint distances of the vectors with magnitude one.
\begin{Theorem}{\cite[Theorem 3.1]{wirth2009edge}}\label{thm:cayleymenger} A list of six positive numbers $E = (a,b,c, x, y, z)$ form the edge lengths of a tetrahedron - with faces $(a,b,z)$, $(a, y, c)$, $(x,b,c)$ and $(x,y,z)$ - exactly when 
\begin{enumerate}
    \item $\min(a+b+c,a+y+z,x+b+z,x+y+c)>max(a+x,b+y,c+z)$, and 
    \item $D(E)$ is positive, where $D(E)$ is the so called Cayley-Menger determinant\footnote{The term Cayley-Menger determinant was coined by Blumenthal in his classic book on Euclidean distance geometry~\cite{blumenthal1953theory}.}
    \begin{equation} D(E) = \begin{vmatrix} \label{eq.cayleymenger}
        0 & a^2 & b^2 & c^2 & 1 \\
        a^2 & 0 & z^2 & y^2 & 1\\
        b^2 & z^2 & 0 & x^2 & 1\\
        c^2 & y^2 & x^2 & 0 & 1 \\
        1 & 1 & 1 & 1 & 0 
    \end{vmatrix} \end{equation}
\end{enumerate}
\end{Theorem}

With this theorem in place, we obtain a way to bound $d_{ij}$ from above and below. 

\begin{Lemma}\label{lem:cayleymengerlemma}
    Let $S$ be the support of $x \in U^\ell \subseteq \mathcal{V}_{m,n}$ with known distance data $D=D(1,1)\times D(1,r)$. For any $i,j<m$,  $\alpha_{ij} \leq d_{ij} \leq \beta_{ij}$ where $\alpha_{ij}, \beta_{ij}$ are the two positive roots 
    of the Cayley-Menger polynomial $f(z) = D(1,1,r,d_{im}, d_{jm}, z)$.
\end{Lemma}

\begin{proof}
By definition the polynomial $f(z) = D(1,1,r,d_{im},d_{jm}, z)$ is biquadratic in the variable $z$. Moreover, $|r-1| \leq d_{im}, d_{jm} \leq r+1$ which forces the biquadratic form $f(z)$ to have four real roots (see Proposition~\ref{prop:cayleymengerprop} in Appendix \ref{ap:biquadratic}). As such, $f(z)$ admits exactly two positive roots, which are precisely the desired upper and lower bounds on the distance $d_{ij}$ which ensure that $E = \{1, 1,r, d_{im}, d_{jm}, d_{ij}\}$ constitutes the set of edge lengths of a tetrahedron.
\end{proof}

For each $(i,j)$ with $i,j<m$, the pre-processing step, outlined in Algorithm \ref{alg:preprocessing} below, uses the bounds from Lemma~\ref{lem:cayleymengerlemma} to restrict the number of possible values the Gram matrix entry $A_{ij}$ may take. It achieves this by constructing an ambiguity list for each pair $(i,j)$ which contains all admissible values of $A_{ij}$, namely those inner products $\ip{v_i}{v_j}$ whose associated distances $d_{ij}$ satisfy $\alpha_{ij} \le d_{ij} \le \beta_{ij}$. This reduces the search space by restricting every $A_{ij}$ to a limited set of possibilities governed by the Cayley-Menger bounds, with some $(i,j)$ typically having very substantial restriction in the number of possibilities (see Figure~\ref{fig:minvalid}).

\begin{algorithm}[h]
\caption{Pre-Processing step via Cayley-Menger determinants}\label{alg:preprocessing}
\KwIn{Number of points $m$, dimension $n$, second moment $m_2(x) = \t(X^\top X) = \{(a_{ii}, a_{jj}, a_{ij})\}_{i,j =1}^m$
where $a_{ii} = \|v_i\|^2$, $a_{jj} = \| v_j \|^2$ and $a_{ij} = \langle v_i , v_j \rangle$.}
\KwOut{Dictionary $\valid$ where $\valid[(i,j)] = $ \{feasible $a_{ij}$ for the $(i,j)$ coordinate of $X^\top X$\}.}
\BlankLine 
\BlankLine 

Set $D = \{d_{ij}^2\}_{i,j=1}^m$ where $d^2_{ij} = a_{ii} + a_{jj} - 2a_{ij}$ \Comment{Note: $D$ is a set of unlabeled distances}

\BlankLine

\ForAll{$(i,j)$, $i,j < m$}{
  Initialize \valid[$(i,j)$] $\gets [ \, ]$

  $\alpha, \beta \gets \texttt{ComputeCayleyMengerBounds(}a_{ii}, a_{jj}, a_{mm}, d_{im}, d_{jm}\texttt{)}$
  
  Set \valid[$(i,j)$] $\gets [a_{ij} \, | \, \alpha \leq d_{ij} \leq \beta]$
  
}
\textbf{Output} \valid

\BlankLine 
\BlankLine 

\SetKwFunction{FMain}{ComputeCayleyMengerBounds}
    \SetKwProg{Fn}{Function}{:}{}
    \Fn{\FMain{$a_{ii}, a_{jj}, a_{mm}, d_{im}, d_{jm}$}}{
    \BlankLine
    Build the Cayley-Menger matrix
 
    $$M = \begin{bmatrix}
        0 & a_{ii} & a_{jj} & a_{mm} & 1 \\
        a_{ii} & 0 & z^2 & d_{im}^2 & 1 \\
        a_{jj} & z^2 & 0 & d_{jm}^2 & 1\\
        a_{mm} & d_{im}^2 & d_{jm}^2 & 0 & 1\\
        1 & 1 & 1 & 1 & 0
    \end{bmatrix}$$

    Set $p(z) \gets \det(M)$ \Comment{$\deg(p(z)) = 4$ and has two roots $> 0$}
    
    Solve $p(z) = 0$ to find the two positive roots $\alpha, \beta$ (with $\alpha < \beta$)}
    
    \textbf{return} $\alpha, \beta$

\end{algorithm}

\subsubsection{Main algorithm}

With the pre-processing step in place, we are now ready to describe the main algorithm. The algorithm takes as input the dictionary $\valid$ produced in the previous stage, where for each pair $(i,j)$, $\valid[(i,j)]$ encodes the list of admissible values for the Gram matrix entry $A_{ij}$. The key idea behind the algorithm is to exploit Theorem \ref{thm:mainthm}, which guarantees that for any choice of indices $\iota = (i_1,\dots, i_{n})$ with $i_1,\dots, i_n < m$, there is a unique subset of $\t(A)$ and a unique arrangement of values $\ip{v_i}{v_j}$ in this subset into an $n \times n$ matrix $P$ such that 
$$\mathrm{rank} \begin{bmatrix}
    P & \beta \\
    \beta^T & r
\end{bmatrix} \leq n,$$
where $\beta^T = [\ip{v_{i_k}}{v_m}]_{k=1}^n$. This unique matrix $P$ is, by construction, a principal submatrix of $A$. Proceeding iteratively, we recover the Gram matrix $A$ once all principal submatrices are exhausted. To optimize for efficiency, we maintain three sets: 
\begin{enumerate}
    \item $\used$, to store fixed entries $a_{ij}$ of $A$
    \item $\usedind$, to store fixed indices $(i,j)$ of $A$
    \item $\counts$, indexed by $n-$tuples $\iota = (i_1,\dots, i_n)$, to store the total number of rank checks 
    \begin{equation}
    c_{\iota} = \prod_{1\leq r<s \leq n} |\valid[(i_r,i_s)]|        
    \end{equation}
required to determine the correct principal submatrix indexed by $\iota$. 
\end{enumerate}

In each iteration, the algorithm first chooses the index $\iota = (i_1,\dots, i_n)$ that minimizes $\counts$. Subsequently, it identifies the unique tuple $(a_{i_ri_s}^* \, | \, i_r<i_s \, i_r,i_s \in \iota)$ in 
\begin{equation}
\mathbb{V}_\iota = \prod_{1\leq r<s \leq n}\valid[(i_r,i_s)]    
\end{equation}
 so that the $(n+1) \times (n+1)$ matrix $M_{\iota,m} = \mathrm{diag}(\|v_{i_r}\|^2 \, | \, i_r \in (\iota, m)) + \mathrm{Matrix}(a_{i_ri_s}^* \, | \, i_r,i_s \in (\iota, m))$ has rank at most $n$. By theorem~\ref{thm:mainthm}, $M_{\iota,m}$ has to be the principal submatrix of $A$ indexed by $(\iota, m)$. The algorithm now fixes these entries in $A$. The resolved entries are then removed from $\t(A)$, added to $\used$, and their indices transferred to $\usedind$. Next, $\valid[(i_r,i_s)]$ is collapsed to the correct value $A_{i_ri_s}$ for each $(i_r,i_s) \in \iota \times \iota$. Finally, $\counts$ is updated to reflect the reduced number of checks required to determine the remaining terms of $A$. Algorithm \ref{alg:main} outlines this procedure below.

\begin{algorithm}[h]
\caption{Recovering support $X \in \V_{m,n}$ of a $\delta-$function $x$}\label{alg:main}
\KwIn{Dictionary $\valid$ generated in algorithm \ref{alg:preprocessing}, number of points $m$, vector $\alpha$ and diagonal of norms $d$.}
\KwOut{Support configuration $X$ (upto orthogonal translate).}

\BlankLine
\textbf{Step 1: Initialize}
\BlankLine
Set $G = \mathrm{diag}(d)$

Set last row and column of $G$ to $[\alpha, \, \|v_m\|^2]^T$.

Initialize $\used \gets \emptyset$ 

Initialize $\usedind \gets \emptyset$

Initialize $\counts \gets \{\}$ \Comment{$\counts$ is a dictionary}
\BlankLine
\textbf{Step 2: Iterative assembly}
\BlankLine

Set $\counts \gets \texttt{UpdateCounts}(\counts, \, \usedind)$

\While{$\counts \neq \varnothing$}{

Set $\iota \gets \arg \min(\counts[\iota'])$
  
  \ForEach{$(a_{i_ri_s} | i_r < i_s, \, i_r, i_s \in \iota) \in \mathbb{V}_\iota$}{
  \BlankLine
    Set $M \gets \mathrm{diag}(d_{i_r} \, | \, i_r \in (\iota,m)) + \mathrm{Matrix}(a_{i_ri_s} | \, i_r,i_s \in (\iota, m)) $
    \BlankLine
    \If{$\mathrm{rank}(M) \leq n$}{
      Set $G_{ij} \gets a_{ij}$, $G_{ik} \gets a_{ik}$, $G_{jk} \gets a_{jk}$

      Update $\used$,  $\usedind$,  $\valid$
      
      Set $\counts \gets \texttt{UpdateCounts}(\counts, \, \usedind)$ \Comment{Note: $(i,j,k) \not\in \counts$ after update}
      
      \textbf{break}
    }
  }
}
\BlankLine 
\textbf{Output }$\texttt{Cholesky}(G)$.

\BlankLine 
\BlankLine

\BlankLine

\SetKwFunction{FMain}{UpdateCounts}
    \SetKwProg{Fn}{Function}{:}{}
    \Fn{\FMain{$\counts, \usedind$}}{
    \BlankLine
    \ForEach{$(i,j,k) \not\in \usedind$}
    {Set $\counts[(i,j,k)] \gets  c_\iota$} 
    
    \textbf{return} Counts
}
\BlankLine

\end{algorithm}

\subsubsection{Discussion}\label{sec:discussion}

We dedicate this section to analyzing the time complexity of Algorithms \ref{alg:preprocessing} and \ref{alg:main}. The complexity of Algorithm \ref{alg:preprocessing} follows directly from its structure. The algorithm iterates over all pairs $(i,j)<m$, amounting to $\binom{m}{2}=O(m^2)$ iterations. Since each evaluation of \texttt{ComputeCayleyMengerBounds} runs in constant time and the insertion of admissible values into $\valid[(i,j)]$ can be performed in $O(\log m)$ time (via binary search), the overall running time is $O(m^2\log m)$. 

Ignoring the effects of dynamic updates, we can derive a weak upper bound for the time-complexity of Algorithm \ref{alg:main}. The algorithm proceeds iteratively by picking the $n-$tuple $\iota$ requiring the minimum number of rank checks - which can be performed in $O(n\log(m))$ time using a binary search. Once the tuple $\iota$ has been determined, the number of candidates (for the correct minor $M$) that must be checked is given by
$$c_{\iota} = \prod_{\substack{i_r, i_s \in \iota \\ i_r < i_s}} |\valid[(i_r, i_s)]|.$$ In the worst case—if the pre-processing step fails entirely—each $\valid[(i_r,i_s)]$ has size $O(m^2)$, yielding $c_\iota = O((m^{2{\binom{n}{2}}}) = O(m^{n(n-1)})$ candidates for the correct principal submatrix $M_{\iota, m}$. Since verifying the rank of an $(n+1)\times(n+1)$ matrix requires $O(n^3)$ operations, each $n$-tuple $\iota$ requires $O(n^3 m^{n(n-1)})$ computations. Finally, upon identifying the correct $M_{\iota, m}$, we update the dictionaries $\valid$ and $\counts$. Since they both contain on the order of $\binom{m}{n} \sim m^n/n!$ elements, updates to them can be made in $O(m^n)$ time. Thus, each iterate of our algorithm can be performed in $O(n\log(m) + n^3m^{n(n-1)} + 2m^n) = O(n^3m^{n(n-1)})$ time.

Finally, if we assume $\binom{n}{2}$ entries of $A$ indexed by $\iota$ have already been determined in any first iteration, then the next $n-$tuple of indices $\iota'$ that minimizes $\counts$ will have to differ from $\iota$ by a single index, say $i_s$. As a result, there are precisely $n-1$ new entries of the Gram matrix $A$ that will be determined in the minor involving $\iota'$: namely $a_{i_ri_s}$ for all $r\neq s \in \iota'$. This pattern persists until all principal minors are exhausted. Consequently, the full determination of $A$ requires a total of $O(m^2/n)$ iterations. Combining these estimates, we see that the overall complexity is bounded from above by $O(n^2m^{n(n-1)+2})$. \\

\subsubsection{Numerical Experiments}

From the above complexity analysis, it is evident that much of the burden in the theoretical bound for Algorithm \ref{alg:main} stems from the assumption that the pre-processing step completely fails — a scenario that rarely occurs in practice. For example
for random configurations in $\V^{m-1}_{m,n}$ with $m>n$ where $|v_i| = 1$ for $i < m$ and $|v_{m}| = 2$, interpoint distances in $D(1,1)$ concentrate in the interval $[1,2]$ due to spherical sampling constraints. In particular, the inner products are typically small $\leq 0.5$ forcing distances in $D(1,1)$ to satisfy $2 \geq d_{ij}^2 = 2 - 2\ip{v_i}{v_j} \geq 1$. This effect is illustrated in Figure \ref{fig:distvals}, showing the distance distribution for a cloud of $100$ uniformly distributed points on the unit sphere in $\r^n$ for $n=3,5,10,20$ where we take $v_m = 2e_1$\footnote{Note
that we can always appply an orthogonal transformation to ensure that $v_m = r e_1$ where $r= |v_m|$.} 

\begin{figure}[h]
    \centering
    \begin{subfigure}[t]{0.23\textwidth}
        \centering
        \includegraphics[scale = 0.24]{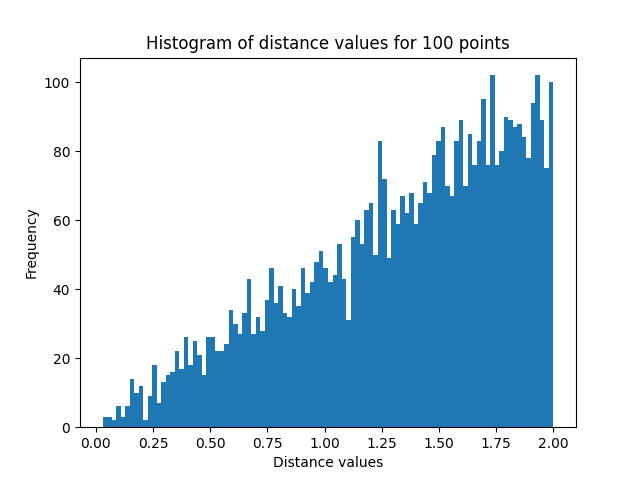}
        \caption{$n=3$}
        \label{fig:distvals3}
    \end{subfigure}%
    ~ 
    \begin{subfigure}[t]{0.23\textwidth}
        \centering
        \includegraphics[scale=0.24]{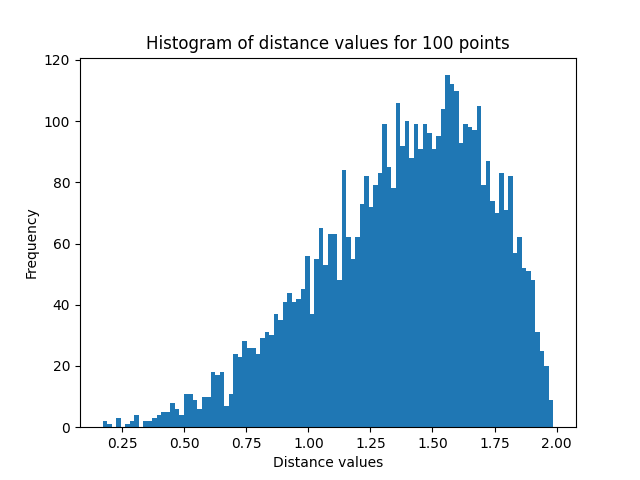}
        \caption{$n=5$}
        \label{fig:distvals5}
    \end{subfigure}
    ~
    \begin{subfigure}[t]{0.23\textwidth}
        \centering
        \includegraphics[scale=0.24]{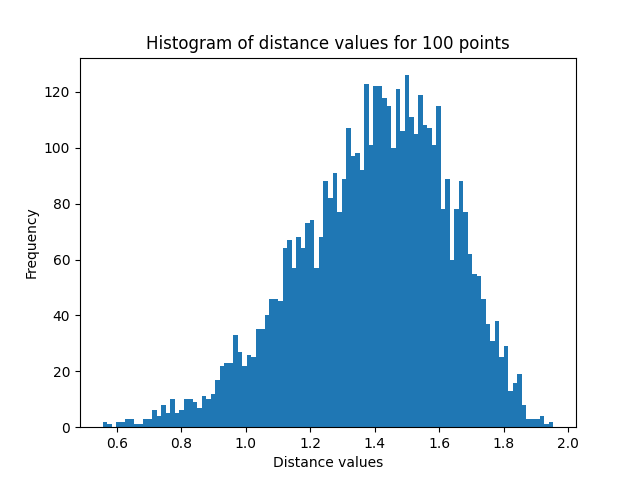}
        \caption{$n=10$}
        \label{fig:distvals10}
    \end{subfigure}
    ~
    \begin{subfigure}[t]{0.23\textwidth}
        \centering
        \includegraphics[scale=0.24]{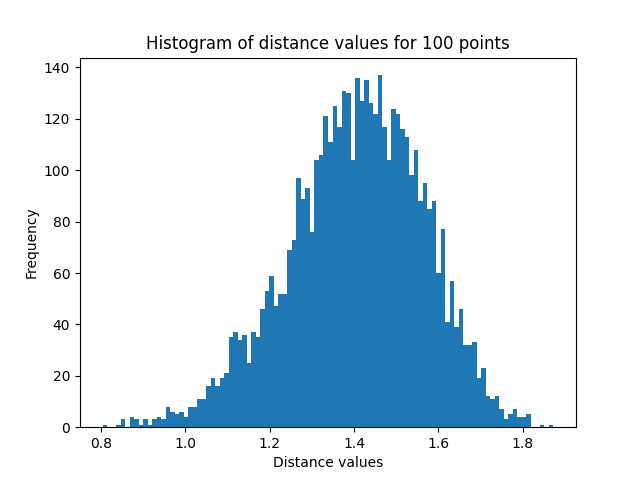}
        \caption{$n=20$}
        \label{fig:distvals20}
    \end{subfigure}
    \centering
    \caption{The distance distributions of a point-cloud with $100$ points, uniformly distributed in the sphere $S^{n-1}$ for $n=3,5,10,20$ respectively.}
    \label{fig:distvals}
\end{figure}
Given such a distance distribution over $[0,2]$, the Cayley–Menger bounds $\alpha_{ij}$ and $\beta_{ij}$ derived in Lemma~\ref{lem:cayleymengerlemma} further confine each $d_{ij}$ to a narrow sub-interval of $[0,2]$. In particular, distances in $D(1,1)$ that appear near the extremities of this range are forced to admit only a small number of feasible alternatives within the Cayley-Menger bounds, sharply reducing the effective search space. This fact paired with dynamic updates to the Gram matrix effectively forces the algorithm to converge to the correct Gram matrix $A$ rapidly - much faster than the naive upper bound dervied in Section~\ref{sec:discussion}. To empirically verify this, we performed the following numerical experiment:\\

For $m=4,\dots, 10$ we generated a total of $100$ random point clouds in $\r^3$ consisting of $m-1$ points distributed uniformly randomly on the unit sphere and the remaining point fixed at $2e_1$. Passing the second moment data $\t(A)$ of each of the point clouds into the pre-processing step (Algorithm \ref{alg:preprocessing}) we obtained ambiguity lists for each pair $(i,j)$ with $i,j < m$. Figure \ref{fig:minvalid} presents a plot comparing the average size of the smallest ambiguity list across all $100$ trials (for each $m$) and the theoretical worst-case scenario where the pre-processing step fails completely and each ambiguity list has size $\binom{m}{2}$. From this figure, it is clear that the pre-processing step \ref{alg:preprocessing} significantly narrows the search space, indicating the theoretical time complexity bound for Algorithm \ref{alg:main} derived in Section~\ref{sec:discussion} is overly conservative. The true performance of Algorithm~\ref{alg:main} is further illustrated in Figure \ref{fig:mincounts} which compares the average number of rank checks (across all $100$ trials) needed to recover the Gram matrix $A$, and the worst case theoretical upper bound derived in Section~\ref{sec:discussion}. In particular, we notice that the actual number of rank checks (on average) needed before convergence is about five orders of magnitude smaller than the estimated theoretical upper bound! This improvement is due to both the effectiveness of the pre-processing step and the dynamic updates performed to the Gram matrix, which progressively reduces the set of admissible principal submatrices and, consequently, the number of required rank checks. To illustrate this reduction more explicitly, consider a randomly selected representative trial (among the $100$) with $m=6$. After pre-processing, the initial dictionary of $\counts$ was determined to be 

\[
\counts = \left\{
\begin{array}{ll}
(0,1,2): 270, & (0,2,4): 270, \\
(0,1,3): 360, & (0,3,4): 324, \\
(0,1,4): 360, & (1,2,3): 720, \\
(0,2,3): 240, & (1,2,4): 810, \\
              & (1,3,4): 900, \\
              & (2,3,4): 648
\end{array}
\right\}
\]
The algorithm selected $\iota=(0,2,3)$ (the triple with the fewest rank checks needed, $240$), found the correct principal submatrix of $A$ indexed by $(\iota,m)$ and updated $\valid$ after fixing $A_{02},A_{03},A_{23}$. This reduced $\counts$, in the very next iterate, to
\[
\counts = \left\{
\begin{array}{ll}
(0,1,2): 28,  & (1,2,3): 56,  \\
(0,1,3): 32,  & (1,2,4): 392, \\
(0,1,4): 128, & (1,3,4): 448, \\
(0,2,4): 28,  & (2,3,4): 49,  \\
(0,3,4): 28   &
\end{array}
\right\}
\]
Three iterations later, the dictionary had shrunk to mostly single-digit counts. The algorithm successfully converged converged to the correct Gram matrix within $6$ iterations in total, staying well below our expectation of $6^2/3 = 12$ total iterations. Finally, Figure \ref{fig:numiter} confirms the expected trend of $O(m^2/n)$ iterations needed for convergence to the correct Gram matrix $A$.\\

\begin{figure}[h]
    \centering
    \begin{subfigure}[t]{0.31\textwidth}
        \centering
        \includegraphics[scale = 0.32]{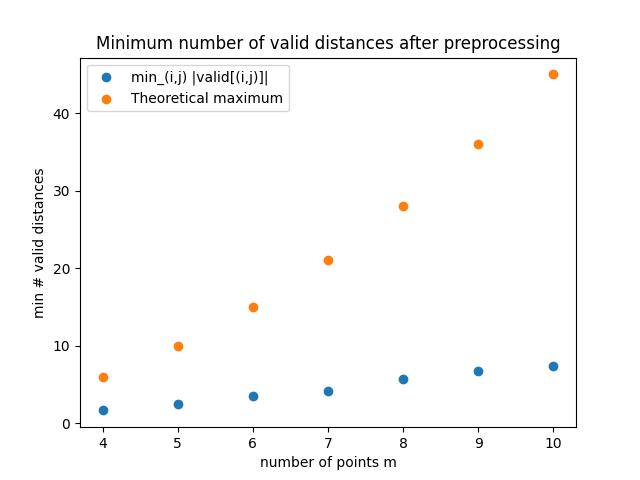}
        \caption{Size of smallest ambiguity list after the pre-processing  step vs. theoretical maximum $y(m) = \binom{m}{2}$}
        \label{fig:minvalid}
    \end{subfigure}%
    ~ 
    \begin{subfigure}[t]{0.31\textwidth}
        \centering
        \includegraphics[scale=0.32]{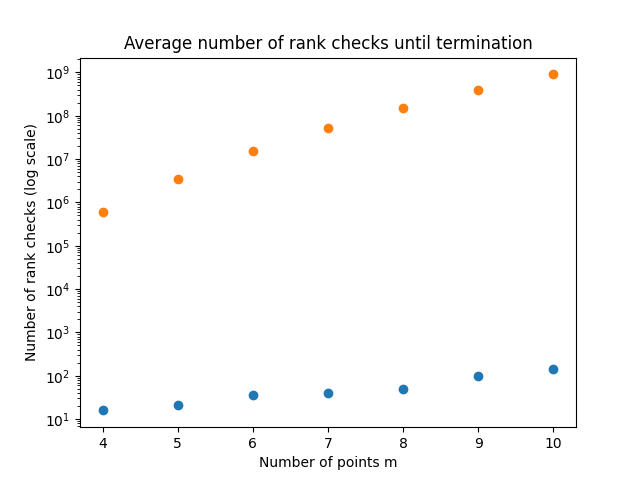}
        \caption{Average number of rank-checks needed to converge to the right Gram matrix vs. the theoretical maximum $y(m) = n^2m^{n(n-1)+2}$}
        \label{fig:mincounts}
    \end{subfigure}
    ~
    \begin{subfigure}[t]{0.31\textwidth}
        \centering
        \includegraphics[scale=0.32]{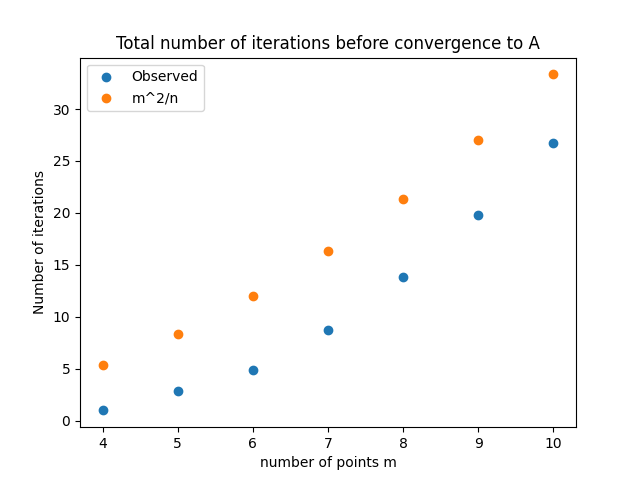}
        \caption{Average number of observed iterations before convergence vs. Expected number of iterations before convergence}
        \label{fig:numiter}
    \end{subfigure}
    \caption{Comparison of the observed performance of Algorithms \ref{alg:preprocessing}  and \ref{alg:main} vs. the theoretical upper bounds derived in Section \ref{sec:discussion}}
\end{figure}

These trends remain true in higher dimensions. Repeating the experiment with $100$ random point clouds in $\r^4$ for $m=5,6,7,8$, we observe from Figure \ref{fig:mincountsdim4} that convergence requires more rank checks than in three dimensions, reflecting the fact that each $4$-tuple $\iota$ involves $\binom{4}{2}=6$ coordinates $(i,j)$ of $A$ that need to be determined. For instance, even if the pre-processing step (Algorithm \ref{alg:preprocessing}) reduces each ambiguity list is to $\sim 10$ candidates, the resulting search in the main Algorithm \ref{alg:main} may demand $10^6$ rank checks. Nevertheless, Figure \ref{fig:mincountsdim4} demonstrates that, on average, the actual number of checks remains about five orders of magnitude below the theoretical upper bound of $O(m^{12})$. Finally, Figure \ref{fig:numiterdim4} shows that, owing to the reduced volume of high-dimensional spheres and the resulting constraints on inner products, the algorithm typically outperforms the expected $O(m^2/n)$ iteration bound, though the overall growth trend remains consistent.

\begin{figure*}[h]
    \centering
    \begin{subfigure}[t]{0.48\textwidth}
        \centering
        \includegraphics[scale=0.4]{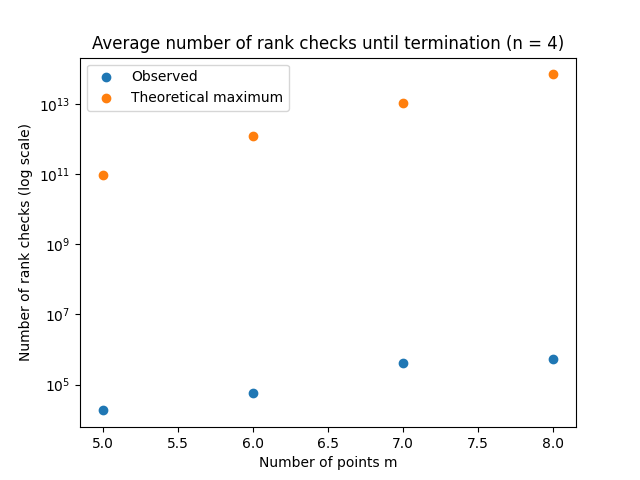}
        \caption{Average number of rank-checks needed to converge to the right Gram matrix vs. $y(m) = n^2m^{n(n-1)+2}$}
        \label{fig:mincountsdim4}
    \end{subfigure}
    ~
    \begin{subfigure}[t]{0.48\textwidth}
        \centering
        \includegraphics[scale=0.4]{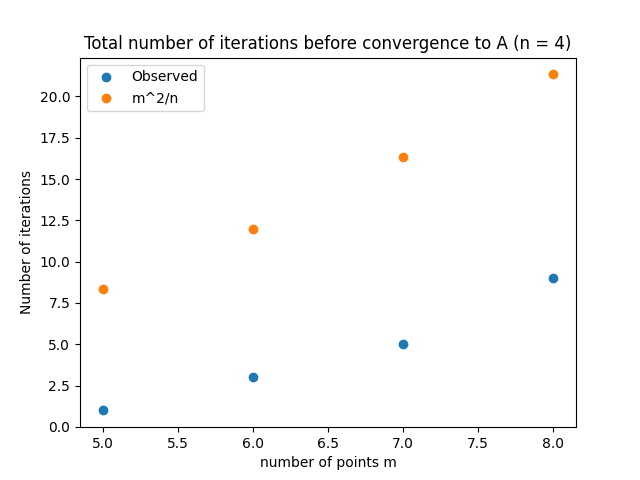}
        \caption{Average number of iterations before convergence}
        \label{fig:numiterdim4}
    \end{subfigure}
    \caption{Reconstructing Gram matrix $A = X^\top X$ using algorithm $\ref{alg:main}$ in $\r^4$}
\end{figure*}

\subsection{Signals with support on the unit sphere $S^{n-1}$}\label{sec:unitsphere}
This case is considerably more subtle. Given a binary $\delta$-function $x$ with support $S = \{v_1, \ldots , v_m\} 
\subset S^{n-1}$, the Gram matrix $A = X^\top X$ (where $X = [v_1 \ldots v_s]$ for any ordering of the points of $S$) admits $m!$ equivalent symmetric rearrangements. Thus we can deduce that there are $m!$ ways to arrange he inner products $\ip{v_i}{v_j}$ into an $m\times m$ positive-semidefinite matrix $M$ of rank at most $n$. Indeed, Theorem~\ref{thm:mainthm}, says that for generic support vectors $v_1, \dots, v_m\in S^{n-1}$ these are the only ways to obtain an $m \times m$ symmetric psd matrix from the set of inner products $\langle v_i, v_j \rangle$. Guided by this, we seek to adapt Algorithm~\ref{alg:main} to this setting as well.  Certainly, we no longer have access to the preprocessing step (Algorithm~\ref{alg:preprocessing}) since it relied crucially on distances to the isolated point. Thus, we are forced to consider Algorithm~\ref{alg:main} directly. The main changes are as follows: \\

Instead of considering index lists $\iota$ of size $n$, we now consider $\iota$ to be index lists of size $n+1$. We define the helper dictionaries $\valid$, $\used$, $\usedind$ and $\counts$ as before. We then initialize $G = \mathrm{diag}(\mathbbm{1}_m)$ and for each $(i,j)$  we set $\valid[(i,j)]$ equal to the list of all possible entries $[\ip{v_i}{v_j} | i<j]$. During each iterate we identify $\iota = \arg\min \counts[\iota']$ and for each
$$(a_{i_ri_s} | i_r<i_s,\, \, i_r,i_s \in \iota) \in \prod_{\substack{(i_r,i_s) \in \iota\\i_r<i_s}} \valid[(i_r,i_s)]$$
we set $M= \mathrm{Matrix}(a_{i_ri_s} | i_r,i_s \in \iota)$ before performing the rank check and proceeding iteratively as before by updating $\valid$ and $\counts$. 

The absence of the pre-processing step is felt through the unavoidable computational cost of $O(n^3m^{2{\binom{n+1}{2}}}) = O(n^3m^{n^2+n})$ in the initial iterations, where we have to iterate through all possible candidates for $M_\iota$ and perform a rank check. As such, in the initial stages, our algorithm performs the same order of computations as the time-complexity upper bound discussed in Section~\ref{sec:discussion}. However, just as before, the number of rank-checks needed in the later iterations go down drastically because of the dynamic updates to the Gram matrix. Moreover, with each correct minor being fixed in place, we are also implicitly fixing some component of the $S_m$ labeling ambiguity. Thus when enough minors are fixed in place, there is precisely one way to put the rest of the Gram matrix $X^\top X$ together. Once again, when a principal submatrix indexed by columns $\iota$ is fixed in place in the overall Gram matrix $X^\top X$, any subsequent index $\iota'$ that minimize counts will necessarily have non-trivial intersection with $\iota$ by design. Thus the algorithm naturally traverses its search space in a `depth-first fashion' resolving minors whose column indices overlap with those that have already been fixed in a previous iterate. This process continues until no such minor exists and the algorithm proceeds to the next branch. Running this adapted version of the algorithm for $m=6,\dots, 10$ always converges to the correct Gram matrix by design. The overall performance of this adaptation across all trials is illustrated in Figure~\ref{fig:rankchecksunitsphere} below.
    \begin{figure}[h]
        \centering
        \includegraphics[width=0.5\linewidth]{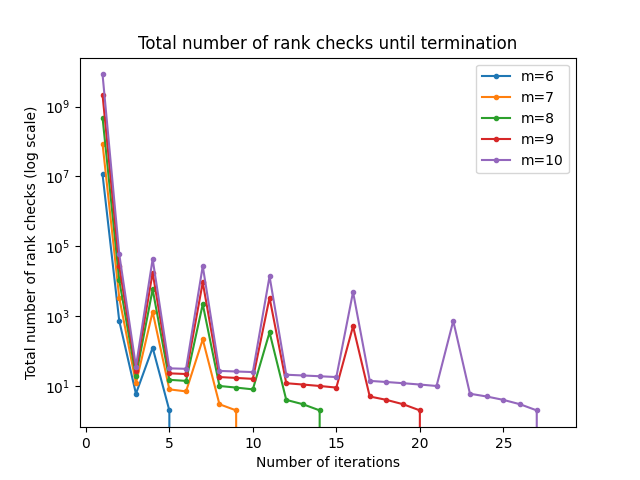}
        \caption{Number of rank checks needed on average to fix a principal submatrix $M_\iota$ vs. Number of iterations}
        \label{fig:rankchecksunitsphere}
    \end{figure}

%While we expect an analogue of Theorem~\ref{thm:mainthm} to hold when all support vectors have the same norm, our original proof does not directly extend to this case because it relies crucially on the presence of an isolated point whose inner-products with the remaining support vectors allowed us to label the remaining support vectors. When all vectors are constrained to lie on a common sphere, the second moment reduces to having a set of unlabeled inner-products $\ip{v_i}{v_j}$ and we no longer have the ability to perform such a labeling. \\

%Though this proof is currently unavailable, we have empirically verified this assertion for $n=3,4,5,6$ and $m=n+1,\dots, 7$. We achieved this by performing $100$ trials of uniformly sampling $m$ points $X \subseteq S^{n-1}$ and exhaustively checking all $\binom{m}{2}$ symmetric rearrangements of $X^\top X$ and observed that there were precisely $m!$ candidates with rank $\leq n$. This computational evidence served as further proof that the desired result must hold true in this case as well. 

%Guided by this evidence, we sought to adapt algorithm \ref{alg:main} to this setting as well.

\subsection{Recovering generic $x$ from noisy $m_2(x)$ measurements}\label{sec:noisyalg}
In this section we return to the case where $x\in \V_{m,n}^{m-1}$ and consider the situation where the second moment measurements $\t(A)$ \textit{are noisy}. As such, we seek to adapt Algorithm~\ref{alg:main} to produce the best approximation $G_X$ of the original Gram matrix $X^\top X$ from a
set of noisy perturbations of the triples in $\t(A)$. 
In our model we assume that the norms of the vector $v_i$ are unchanged but the inner products $\langle v_i, v_j \rangle$ are corrupted by noise. From a distance geometry perspective this leads to a relaxed formulation of the noisy unlabeled distance geometry problem in which we are given an unlabeled collection of noisy distances between points $v_i$ and $v_j$ for $i,j \in [m-1]$, together with a labeled set of noisy distances from each $v_i$, $i \in [m-1]$, to the isolated point $v_m$. 
%
%For such a reconstruction problem to be well posed, we require that if the perturbation alters the norm of a support vector $v$ in one triple of the second moment $\t(A)$, then it must alter it in the same way across all triples containing $\|v\|$. Given this requirement, we observe that if the noise alters the norm of a support vector $v$ in the fashion described above, then the corresponding quantities $\ip{v}{v_j}$ can be readily isolated by identifying triples of the form $(\|v\|^2, \|v_j\|^2, \ip{v}{v_j})$ from the noisy second moment $\t^{\text{noisy}}(A)$, leading to a simplification in the reconstruction problem. Thus, proceeding in a similar fashion as in Section~\ref{sec:relaxationtom-1}, we may reduce, without loss of generality, to the setting where the added noise provides no additional labeling information about the inner products $\ip{v_i}{v_j}$. Namely, we restrict our attention to the case where the noise perturbs only the inner products $\ip{v_i}{v_j}$, while leaving the norms $\|v_i\|^2$ unchanged for all $i,j \in [m]$. \\
%As before, the noisy unlabeled inner products $\ip{v_i}{v_j}^{\text{noisy}}$, now give rise to noisy interpoint distances via $\|v_i\|^2 + \|v_j\|^2 - 2\ip{v_i}{v_j}^{\text{noisy}}$. 
%This leads to a relaxed formulation of the noisy unlabeled distance geometry problem in which we are given an unlabeled collection of noisy distances between points $v_i$ and $v_j$ for $i,j \in [m-1]$, together with a labeled set of noisy distances from each $v_i$, $i \in [m-1]$, to the isolated point $v_m$. 

Despite its similarity to the setup in the noiseless case, we cannot apply Algorithm~\ref{alg:main} because, the availability of exact distances is crucial to the success of the algorithm. More precisely Algorithm~\ref{alg:main} always converges to the correct Gram matrix by iteratively identifying $(n+1)\times(n+1)$ principal submatrices of $A$ through a rank check. However, in the presence of measurement noise, the $(n+1)\times(n+1)$ principal submatrices will (with probability one) have full rank even with the correct arrangement of entries. Thus, a strict rank test no longer identifies the correct configuration. 

To extend Algorithm \ref{alg:main} to noisy input data, we use the Eckart–Young theorem which states that the $k$th largest eigenvalue of a symmetric matrix quantifies the distance from that matrix to the variety of matrices with rank $\leq k$. 
Thus, a natural extension of Algorithm \ref{alg:main} to the noisy case is obtained by replacing the rank check in step 2, with the computation of $\lambda_{n+1}(M_{\iota, m})$ which computes the distance of $M_{\iota,m}$ to the rank $\leq n$ locus. Among all candidate principal submatrices, the candidate minimizing this distance is chosen and fixed in the overall approximate Gram matrix $G_X$. Proceeding iteratively, we derive the noisy variant of the main algorithm. Algorithm~\ref{alg:noisymain} summarizes the corresponding modifications to step $2$ of Algorithm~\ref{alg:main}.

\begin{algorithm}[h]
\begin{flushleft}
\caption{Recovering approximate Gram matrix $G_X$  from noisy $m_2$. Noisy variant of Algorithm~\ref{alg:main} }\label{alg:noisymain}

\textbf{Step 2: Iterative assembly}
\BlankLine
Set $\counts \gets \texttt{UpdateCounts}(\counts, \, \usedind)$

\While{$\counts \neq \varnothing$}{

Set $\iota \gets \arg \min \counts[\iota]$
  
  \ForEach{$(a_{i_ri_s} | i_r<i_s, \, i_r,i_s\in \iota) \in \mathbb{V}_{\iota}$}{
  \BlankLine
    Set $M_{\iota, m} \gets \mathrm{Matrix}(a_{i_ri_s} | r,s \in \{\iota,m\})$
    \BlankLine
  }
    Set $M_{\iota^*} = \arg\min_{\iota} \lambda_{4}(M_{\iota,m})$
      
    Set $(G_X)_{i_ri_s} \gets a_{i_r^*i_s^*}$

    Update $\used$,  $\usedind$,  $\valid$
      
    Set $\counts \gets \texttt{UpdateCounts}(\counts, \, \usedind)$ \Comment{Note: $(i,j,k) \not\in \counts$ after update}
    }
\BlankLine 
\textbf{Output }$ G$.
\end{flushleft}
\end{algorithm}\,\\
Notice that if the underlying noise-level is zero, Algorithm~\ref{alg:noisymain} recovers the same output as that of the exact version presented in Algorithm~\ref{alg:main}. More precisely, this modification retains the overall structure of the algorithm while improving its robustness to small perturbations, allowing for reliable recovery even when the input data is noisy. When evaluating performance, we expect Algorithm~\ref{alg:noisymain} to be effective in regimes with low noise. In such settings, the correct candidate submatrix $M_{\iota,m}$ should consistently yield the smallest $\lambda_{n+1}$, enabling the correct assembly of $G_X$. However, when the noise is comparable in magnitude to the second moment measurements, multiple candidates may display nearly indistinguishable eigenvalue profiles. Consequently, an incorrect submatrix may exhibit a smaller minimal eigenvalue than the correct one, leading to spurious reconstructions.

\subsubsection{Numerical Experiment} To test the heuristic described above, we conducted the following numerical experiment. We considered $50$ different randomly generated point-clouds, each of $6$ points in $\r^3$. In each cloud, $5$ points were distributed uniformly on the sphere with radius $3$ and the last point was fixed at $4e_1$. To control the amount of imputed noise, we considered $30$ different choices for the variance $\sigma^2 \in [10^{-4}, 1]$ (equally spaced in log-scale) of the additive Gaussian noise. For each $\sigma^2$, we generated the noisy second moment measurements by computing $m_2^{\text{noisy}}(x) = \t(X^\top X + \mathcal{M})$ where $\mathcal{M}$ is an $6 \times 6$ matrix with zeros along the diagonal and its off-diagonal entries are random numbers $m_{ij} \sim \mathcal{N}(0,\sigma^2)$. With the noisy second moment generated, we ran our noisy reconstruction algorithm (Algorithm~\ref{alg:noisymain}) on these measurements to reconstruct the noisy matrix $G_X(\sigma^2) = X^\top X + \mathcal{M}$ which serves as an approximation of the low-rank Gram matrix $X^\top X$. We evaluated the performance of our algorithm by computing the proportion (out of $50$ trials) of successful reconstructions. This numerical experiment resulted in a total of $50 \times 30 = 1500$ total reconstructions. When parallelized across 8 CPU cores (without GPU support), the experiment took approximately $5$ minutes to complete, and the results of this numerical experiment are presented in Figure \ref{fig:noisyalg3} below.
\begin{figure}
    \centering
    \includegraphics[width=0.5\linewidth]{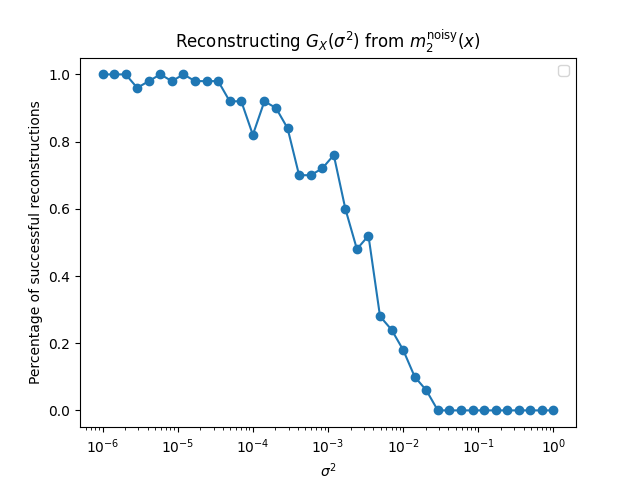}
    \caption{Reconstructing $G_X(\sigma^2)$ from $m_2^{\text{noisy}}(x)$}
    \label{fig:noisyalg3}
\end{figure}
From Figure~\ref{fig:noisyalg3}, it is clear that Algorithm~\ref{alg:noisymain} performs very well in the low-noise regime, achieving near-perfect reconstruction. It continues to perform reasonably well as long as the noise magnitude remains below about $1\%$ of the true (noiseless) measurements. However, its performance deteriorates sharply as $\sigma^2 \to 10^{-1}$, indicating failure once the noise level is about $10\%$ of the magnitudes of the true inner products $\langle v_i, v_j \rangle$. This behavior reveals a key limitation. As the noise level increases multiple candidate submatrices $M_{\iota, m}$ can attain minimal eigenvalues of similar magnitudes. In this regime, an incorrect arrangement of $M$ may end up closer to the rank $\leq n$ locus than the corresponding principal submatrix of $X^\top X$, leading to reconstruction failure. Indeed, all incorrect reconstructions of $G_X(\sigma^2)$ produced by Algorithm~\ref{alg:noisymain} in this experiment, while having a large overall $\lambda_4$, exhibited at least one $4 \times 4$ principal minor whose $\lambda_4$ was smaller than that of the corresponding minor in $X^\top  X$. 

\begin{Remark}
A possible way to improve the performance of Algorithm~\ref{alg:noisymain} is as follows: For each $\iota$, one can short-list and keep track of all candidates $M_{\iota, m}$ that exhibit similar magnitudes in their minimal eigenvalues. Given $\iota'$ (in a subsequent iteration) that differs from $\iota$ by a single element (say $i \in \iota$ and $j \in \iota'$ with $i \neq j$), we filter down the short-lists to only compatible candidate pairs ($M_{\iota,m}$, $M'_{\iota, m}$) such that the entries of $M_{\iota,m}$ and $M'_{\iota, m}$ are equal over the rows and columns indexed by $\iota \cap \iota'$. Among such compatible pairs, we select the pair (along with a value for the entry $M_{ij} \in \valid[(i,j)]$) that minimizes $\lambda_4$ of the $(n+2)\times (n+2)$ matrix indexed by $\iota \cup \iota'$. Once this determination is made, we can fix these entries in $G_X$ and proceed iteratively. Of course, this modification comes at the price of increasing the space and time complexity of the Algorithm.
\end{Remark}

\bibliographystyle{plain}

\begin{thebibliography}{10}

\bibitem{abas2022generalized}
Asaf Abas, Tamir Bendory, and Nir Sharon.
\newblock The generalized method of moments for multi-reference alignment.
\newblock {\em IEEE Transactions on Signal Processing}, 70:1377–1388, 2022.

\bibitem{bandeira2023estimation}
Afonso~S Bandeira, Ben Blum-Smith, Joe Kileel, Jonathan Niles-Weed, Amelia
  Perry, and Alexander~S Wein.
\newblock Estimation under group actions: recovering orbits from invariants.
\newblock {\em Applied and Computational Harmonic Analysis}, 66:236--319, 2023.

\bibitem{bandeira2017optimal}
Afonso~S. Bandeira, Philippe Rigollet, and Jonathan Weed.
\newblock Optimal rates of estimation for multi-reference alignment.
\newblock {\em Mathematical Statistics and Learning}, 2017.

\bibitem{bendory2020single}
Tamir Bendory, Alberto Bartesaghi, and Amit Singer.
\newblock Single-particle cryo-electron microscopy: Mathematical theory,
  computational challenges, and opportunities.
\newblock {\em IEEE Signal Processing Magazine}, 37(2):58--76, March 2020.
\newblock Epub 2020 Feb 27.

\bibitem{bendory2025transversality}
Tamir Bendory, Nadav Dym, Dan Edidin, and Arun Suresh.
\newblock A transversality theorem for semi-algebraic sets with application to
  signal recovery from the second moment and cryo-em.
\newblock {\em Foundations of Computational Mathematics}, sep 2025.

\bibitem{bendory2024sample}
Tamir Bendory and Dan Edidin.
\newblock The sample complexity of sparse multireference alignment and
  single-particle cryo-electron microscopy.
\newblock {\em SIAM Journal on Mathematics of Data Science}, 6(2):254--282,
  2024.

\bibitem{bendory2025orbit}
Tamir Bendory, Dan Edidin, Josh Katz, and Shay Kreymer.
\newblock Orbit recovery for spherical functions.
\newblock {\em arXiv preprint arXiv:2508.02674}, 2025.

\bibitem{bendory2024beltway}
Tamir Bendory, Dan Edidin, and Oscar Mickelin.
\newblock The beltway problem over orthogonal groups.
\newblock {\em Appl. Comput. Harmon. Anal.}, 74:Paper No. 101723, 9, 2025.

\bibitem{bendory2023autocorrelation}
Tamir Bendory, Yuehaw Khoo, Joe Kileel, Oscar Mickelin, and Amit Singer.
\newblock Autocorrelation analysis for cryo-em with sparsity constraints:
  Improved sample complexity and projection-based algorithms.
\newblock {\em Proceedings of the National Academy of Sciences}, 120(18), April
  2023.

\bibitem{Billinge2016Assigned}
Simon J.~L. Billinge, Phillip~M. Duxbury, Douglas~S. Goncalves, Carlile
  Lavor, and Antonio Mucherino.
\newblock Assigned and unassigned distance geometry: applications to biological
  molecules and nanostructures.
\newblock {\em 4OR}, 14(4):337--376, December 2016.

\bibitem{blumenthal1953theory}
Leonard~M. Blumenthal.
\newblock {\em Theory and applications of distance geometry}.
\newblock Oxford, at the Clarendon Press, 1953.

\bibitem{boutin2003reconstructing}
Mireille Boutin and Gregor Kemper.
\newblock On reconstructing {$n$}-point configurations from the distribution of
  distances or areas.
\newblock {\em Adv. in Appl. Math.}, 32(4):709--735, 2004.

\bibitem{boutin2007which}
Mireille Boutin and Gregor Kemper.
\newblock Which point configurations are determined by the distribution of
  their pairwise distances?
\newblock {\em International Journal of Computational Geometry and
  Applications}, 17(1):31--44, 2007.

\bibitem{boutin2020drone}
Mireille Boutin and Gregor Kemper.
\newblock A drone can hear the shape of a room.
\newblock {\em SIAM Journal on Applied Algebra and Geometry}, 4(1):123--140,
  2020.

\bibitem{boutin2024global}
Mireille Boutin and Gregor Kemper.
\newblock Global positioning: The uniqueness question and a new solution
  method.
\newblock {\em Advances in Applied Mathematics}, 160, September 2024.

\bibitem{boutin2025developments}
Mireille Boutin and Gregor Kemper.
\newblock New developments in global positioning.
\newblock {\em Notices of the American Mathematical Society Vol.~72, No.~8},
  September 2025.

\bibitem{connelly2024reconstruction}
Robert Connelly, Steven~J. Gortler, and Louis Theran.
\newblock Reconstruction in one dimension from unlabeled {E}uclidean lengths.
\newblock {\em Combinatorica}, 44(6):1325--1351, December 2024.

\bibitem{doerr2022emerging}
Allison Doerr.
\newblock A dynamic direction for cryo-em.
\newblock {\em Nature Methods}, 19(1):29--29, 2022.

\bibitem{dokmanic2015euclidean}
Ivan Dokmanic, Reza Parhizkar, Juri Ranieri, and Martin Vetterli.
\newblock Euclidean distance matrices: Essential theory, algorithms, and
  applications.
\newblock {\em IEEE Signal Processing Magazine}, 32(6):12--30, 2015.

\bibitem{dokmanic2014localize}
Ivan Dokmanić, Laurent Daudet, and Martin Vetterli.
\newblock How to localize ten microphones in one finger snap.
\newblock In {\em 2014 22nd European Signal Processing Conference (EUSIPCO)},
  pages 2275--2279, 2014.

\bibitem{dokmanic2013echoes}
Ivan Dokmanić, Reza Parhizkar, Andreas Walther, Yue~M. Lu, and Martin
  Vetterli.
\newblock Acoustic echoes reveal room shape.
\newblock {\em Proceedings of the National Academy of Sciences},
  110(30):12186--12191, 2013.

\bibitem{duxbury2016unassigned}
P.M. Duxbury, L.~Granlund, S.R. Gujarathi, P.~Juhas, and S.J.L. Billinge.
\newblock The unassigned distance geometry problem.
\newblock {\em Discrete Applied Mathematics}, 204:117--132, 2016.

\bibitem{gkioulekas2024trilateration}
Ioannis Gkioulekas, Steven~J. Gortler, Louis Theran, and Todd Zickler.
\newblock Trilateration using unlabeled path or loop lengths.
\newblock {\em Discrete \& Computational Geometry}, 71(2):399--441, March 2024.

\bibitem{gortler2019generic}
Steven~J. Gortler, Louis Theran, and Dylan~P. Thurston.
\newblock Generic unlabeled global rigidity.
\newblock {\em Forum of Mathematics, Sigma}, 7:e21, 2019.

\bibitem{huang2021reconstructing}
Shuai Huang and Ivan Dokmanić.
\newblock Reconstructing point sets from distance distributions.
\newblock {\em IEEE Transactions on Signal Processing}, 69:1811--1827, 2021.

\bibitem{kam1980reconstruction}
Zvi Kam.
\newblock The reconstruction of structure from electron micrographs of randomly
  oriented particles.
\newblock {\em Journal of Theoretical Biology}, 82(1):15--39, 1980.

\bibitem{menger1928untersuchungen}
Karl Menger.
\newblock Untersuchungen \"{u}ber allgemeine {M}etrik.
\newblock {\em Math. Ann.}, 100(1):75--163, 1928.

\bibitem{patterson1944ambiguities}
A~Lindo Patterson.
\newblock Ambiguities in the x-ray analysis of crystal structures.
\newblock {\em Physical Review}, 65(5-6):195, 1944.

\bibitem{perry2019sample}
Amelia Perry, Jonathan Weed, Afonso~S. Bandeira, Philippe Rigollet, and Amit
  Singer.
\newblock The sample complexity of multireference alignment.
\newblock {\em SIAM Journal on Mathematics of Data Science}, 1(3):497--517,
  2019.

\bibitem{Takagi1924Algebraic}
Teiji Takagi.
\newblock On an algebraic problem related to an analytic theorem of
  Carath\'eodory and Fej\'er and on an allied theorem of Landau.
\newblock {\em Japanese journal of mathematics: Transactions and abstracts},
  1:83--93, 1924.

\bibitem{toader2023frontier}
Bogdan Toader, Fred~J. Sigworth, and Roy~R. Lederman.
\newblock Methods for cryo-em single particle reconstruction of macromolecules
  having continuous heterogeneity.
\newblock {\em Journal of Molecular Biology}, 435(9):168020, 2023.
\newblock New Frontier of Cryo-Electron Microscopy Technology.

\bibitem{wirth2009edge}
Karl Wirth and André Dreiding.
\newblock Edge lengths determining tetrahedrons.
\newblock {\em Elemente Der Mathematik}, 64:160--170, 12 2009.

\end{thebibliography}

\appendix

\section{Gram-like factorizations of complex symmetric matrices} \label{sec:Takagi}
Every $m \times m$ psd symmetric matrix $A$ of rank $n \leq  m$ has a Gram factorization as $A= X^\top  X$ where $X \in \r^{n \times m}$.
The following lemma shows that, if we consider matrices with complex coefficients then any rank $n$
symmetric $m \times m$ matrix can be factored as $X^\top X$ with $X \in \c^{n \times m}$.
%It is a well known fact that every $m\times m$ real symmetric matrix $A$ is diagonalizable - i.e, it admits a factorization %$A=UDU^T$ with $U \in O(m)$ orthogonal and $D$ a diagonal matrix containing the eigenvalues of $A$. In addition if $A$ is %%positive semidefinite, the eigenvalues are non-negative and consequently $A$ admits a Gram factorization of the form $A=X^\top X$. Over the complex numbers, the situation is more nuanced. For one, not every complex symmetric matrix is diagonalizable, with the standard example being $$M = \begin{bmatrix}
%    1 & i \\
%    i & -1
%\end{bmatrix}$$ with both of its eigenvalues zero. Nevertheless, a generic complex symmetric matrix is diagonalizable. Although diagonalization is only a generic phenomenon over $\c$, every complex symmetric matrix admits an analogue of diagnoalization called the Takagi factorization. Through the Takagi factorization, every complex matrix $A$ can be written in the form $A=UDU^T$ with $U \in \c^{m\times m}$ \textit{unitary}, $D$ an $m\times m$ diagonal matrix with non-negative (real) diagonal entries. The diagonal entries of $D$ represent square roots of the non-negative eigenvalues of the Hermitian positive semidefinite matrix $A^*A$. We can use this fact to show that every complex symmetric matrix admits a Gram decomposition. We show this by showing a slightly stronger result.

\begin{Lemma}\label{lem:takagi}
    Given $n\leq m$, the locus of $m \times m$ complex symmetric matrices with rank $\leq n$ is the image of $\c^{n\times m}$ under the map $\phi: \c^{n\times m} \to \operatorname{Sym}^2(\c^m)$ given by $X \mapsto X^\top X$. 
\end{Lemma}
\begin{proof}
    Let $\phi: \c^{n\times m} \to \operatorname{Sym}^2(\c^m)$ be given by $\phi(X) = X^\top X$. Certainly, we see that for any $X \in \c^{m\times m}$
    $$\rank{X^\top X} \leq \rank{X^\top } = \rank{X} \leq n$$
    Thus $\im(\phi)$ is contained within the rank $\leq n$ locus of $\operatorname{Sym}^2(\c^m)$. Conversely, suppose that $A \in \operatorname{Sym}^2(\c^m)$ has rank $\leq n$. Since $A$ is a complex symmetric matrix, it admits a Takagi factorization of the form $A=U^TDU$ where $U$ is a unitary matrix and $D$ is a diagonal matrix with non-negative entries \cite{Takagi1924Algebraic}. 
    Since the rank of $A$ is at-most $n$, we see that
    $D$ has rank at most $n$, since $U$ is invertible. Hence after permuting the diagonal of $D$ 
    and replacing $U$ with the product of $U$ with a corresponding permutation matrix we may assume
    that the last $m-n$ diagonal entries in $D$ are zero. Let
    $U'$ be the $n \times m$ matrix consisting of the first $n$ rows of $U$. Then, since the last $n-m$ rows
    and columns of $D$ are zero 
    $A = (U')^T D' U'$ where $D'$ is the $n \times n$ matrix obtained deleting the last $n-m$ rows and columns of $D$. 
    Then take $X = \sqrt{D'} U'$. 
\end{proof}

\begin{Remark}
When $A$ is a real symmetric matrix then $A$ is orthogonally diagonalizable and we can take $U$ to be a real orthogonal matrix in the Takagi decomposition. Not all complex symmetric matrices are diagonalizable, but they nevertheless have a Takagi decomposition. For example if $A = \begin{bmatrix}1 & i \\i & -1\end{bmatrix}$ then $A$ is not diagonalizable
but has a Takagi decomposition $A = U^T \begin{bmatrix}2 & 0\\0 & 0\end{bmatrix} U$ where
$U = \frac{1}{\sqrt{2}} \begin{bmatrix}1 & i\\ i & 1\end{bmatrix}$ and $A = \begin{bmatrix} 1\\i\end{bmatrix} \begin{bmatrix} 1 & i \end{bmatrix}$
\end{Remark}
\section{Concerning the $4\times 4$ Cayley-Menger determinant} \label{ap:biquadratic}

Theorem \ref{thm:cayleymenger} gives a criterion for when six positive numbers $E = (a,b,c,x,y,z)$ form the edge-lengths of a tetrahedron. In particular, this criterion involves computing the so called Cayley-Menger determinant
    $$D(E) = \begin{vmatrix}
        0 & a^2 & b^2 & c^2 & 1 \\
        a^2 & 0 & z^2 & y^2 & 1\\
        b^2 & z^2 & 0 & x^2 & 1\\
        c^2 & y^2 & x^2 & 0 & 1 \\
        1 & 1 & 1 & 1 & 0 
    \end{vmatrix}$$

Here we show that when $E = (1,1, r, x, y, z)$ with fixed values of $x,y,r$ such that $r \neq 1$ and $|r-1| < x, y < r+1$ the Cayley-Menger determinant is a biquadratic polynomial in $z$ having four distinct real roots. In particular, there exist precisely two distinct positive roots - which constitute the upper and lower Cayley-Menger bounds extracted in Algorithm \ref{alg:preprocessing}.

\begin{Proposition}\label{prop:cayleymengerprop}
   Let $x,y,r$ be fixed distinct positive numbers such that $r \neq 1$ and $|r-1| < x,y < r+1$.
   Then the biquadratic polynomial $f(z) = D(1,1,r,x,y,z)$  has four distinct real roots.
\end{Proposition}
\begin{proof}
   Given $E = (1,1,r,x,y,z)$ with $r,x,y$  fixed then $D(E) = \alpha z^4 + \beta z^2 + \gamma$
    where
   \begin{equation*}
       \begin{split}
           \alpha(x,y) &= -2r^2\\
           \beta(x,y) &= -2+4r^2-2r^4+2x^2+2r^2x^2+2y^2+2r^2y^2-2x^2y^2\\
           \gamma(x,y) &= -2x^4 + 4x^2y^2 - 2y^4
       \end{split}
   \end{equation*}
   Substituting $t=z^2$ into $D(E)$ it suffices to show that the polynomial  $D_E(t) = \alpha t^2+\beta t+\gamma c$  has two positive real roots $t_1$ and $t_2$. \\

    We begin by analyzing the signs of the coefficients $\alpha,\beta$ and $\gamma$. Certainly $\alpha < 0$ and $\gamma = -2(x^2-y^2)^2 < 0$. 
    %To determine the sign of $\beta$ we first observe that for $i=1,2$, $$|r-1| = |\|v\|-\|v_i\|| \leq \|v_i-v\| \leq \|v_i\|+\|v\| = 1+r$$ which forces $(x,y) \in D(r) = [|r-1|, r+1]^2$ for any fixed $r>0$. 
    We will show that when $r \neq 1$, $\beta (x,y) \geq 0$ when $|r-1| < x,y < r+1$, by minimizing $\beta$ as a function of $x, y$ over the
    the square $[|r-1|,r+1] \times [|r-1|, r+1]$ and showing that the minimum is non-negative. To this end note that 
    $\nabla \beta = \langle 4x(1+r^2-y^2), 4y(1+r^2-x^2)\rangle$ vanishes if and only if $x=y=0$ or $x=y=\sqrt{1+r^2}$. However, since $r \neq 1$ and $x,y > |r-1|$, we obtain $x,y\neq 0$, forcing $x=y=\sqrt{1+r^2}$ to be the only relevant critical point in the interior of 
    the square $[|r-1|, r+1] \times [|r-1|, r+1]$. Any other extrema of $\beta(x,y)$ over the square  occur at its boundary. A similar verification along each edge of the square shows that there are no non-trivial critical points along the boundary, leaving the extreme points along the boundary to occur at its vertices. Evaluating $\beta(x,y)$ at its unique internal critical point and at each of the four vertices of the square
    \begin{equation*}
        \begin{split}
            \beta(\sqrt{1+r^2}, \sqrt{1+r^2}) = r^2(1+r^2) &> 0\\
            \beta(r+1,r+1) = \beta(|r-1|,|r-1|) &= 0\\
            \beta(|r-1|,r+1) = \beta(r+1,|r-1|) = 16r^2 &> 0
        \end{split}
    \end{equation*}
    Thus, by the extreme value theorem, $\beta(x,y) >0$ on the square when $|r-1| < x,y < r+1$ as desired.

Having determined the signs of $\alpha,\beta$ and $\gamma$ - we can derive the desired result via the following observation: the roots $t_1$ and $t_2$, if they exist, must satisfy the Vieta relations $$t_1t_2 = \gamma/\alpha > 0 \text{ and } t_1 + t_2 = -\frac{\beta}{\alpha} \geq 0.$$ From the first equation, we deduce that if $t_1, t_2$ are real then they have the same sign. From the second equation, we deduce that this sign has to be positive. Finally, a routine computation ensures that the discriminant $\Delta$ of $D_E(t)$ is 
positive on the interior of the square $[|r-1|, r+1] \times [|r-1|, r+1]$ so
the two roots $t_1$ and $t_2$ are indeed real and positive Hence the biquadratic polynomial $f(z)$ has two positive real roots.    
   
%    Moreover, the discriminant $\Delta$ only vanishes at the extreme values $(x,y) = (|r-1|, |r-1|)$ and $(x,y) = (r+1,r+1)$ which correspond to $v_1 = v_2$. Thus $\Delta > 0$ in all regions of interest, and as a result, $D(E)$ has four distinct real roots. \arun{[[Arun: Though this doesn't happen in practice, in the extreme case where the discriminant is zero, we have $t_1 = t_2$ and therefore $z$ (corresponding to the missing edge length) is uniquely determined. As such, we can remove this line and instead say in the statement of the proposition ``the biquadratic polynomial $f(z) = D(1,1,r,x,y,z)$ has two positive real roots (counted with multiplicity)".]]}
\end{proof}
\end{document}